\documentclass[reqno,12pt]{amsart}

\usepackage{graphicx, amsmath, amssymb, amscd, amsthm, euscript, psfrag, amsfonts,bm}

\newtheorem{thm}[equation]{Theorem}\newtheorem{Thm}[equation]{Theorem}
\newtheorem{Ex}[equation]{Example}
\newtheorem{Q}[equation]{Problem}

\newtheorem{Con}[equation]{Conjecture}

\newtheorem{cor}[equation]{Corollary}

\newtheorem{dfn}[equation]{Definition}\newtheorem{Def}[equation]{Definition}

\numberwithin{equation}{section}

\numberwithin{equation}{section}

\newcommand{\be}{begin{equation}}

\newcommand{\bH}{\mathbb H}

\newcommand{\q}{\mathbb{Q}}

\newcommand{\e}{\epsilon}
\newcommand{\z}{\mathbb{Z}}
\renewcommand{\q}{\mathbb{Q}}
\newcommand{\N}{\mathbb{N}}
\renewcommand{\c}{\mathbb{C}}
\newcommand{\br}{\mathbb{R}}

\newcommand{{\grinv}}{{\Cal G}^{-r}}

\newcommand{\ba}{\backslash}

\newcommand{\G}{\Gamma}

\newcommand{\Haar}{\operatorname{Haar}}

\newcommand{\Cal}{\mathcal}
\renewcommand{\P}{\mathcal P}

\newcommand{\la}{\langle}
\newcommand{\ra}{\rangle}

\newcommand{\SL}{\operatorname{SL}}

\newcommand{\bp}{\begin{pmatrix}}
\newcommand{\ep}{\end{pmatrix}}

\renewcommand{\bp}{{\rm bp}}

\newcommand{\SO}{\operatorname{SO}}

\renewcommand{\H}{\mathcal{H}}
\newcommand{\T}{\operatorname{T}}

\newcommand{\PSL}{\op{PSL}}

\newcommand{\PS}{\rm{PS}}

\newcommand{\op}{\operatorname}

\newcommand{\BR}{\operatorname{BR}}
\newcommand{\Leb}{\operatorname{Leb}}
\newcommand{\BMS}{\operatorname{BMS}}

\renewcommand{\be}{\begin{equation}}
\newcommand{\ee}{\end{equation}}

\newcommand{\B}{B}

\newcommand{\Res}{\op{Res}}

\begin{document}

\title[Counting]{Harmonic analysis, Ergodic theory and Counting for thin groups}

\author{Hee Oh}
\address{Mathematics department, Yale university, New Haven, CT 06520
and Korea Institute for Advanced Study, Seoul, Korea}

\email{hee.oh@yale.edu}

\begin{abstract}
For a geometrically finite group $\G$ of $G=\SO(n,1)$,
we survey recent developments on counting and equidistribution problems for orbits of $\G$
in a homogeneous space $H\ba G$ where $H$ is trivial, symmetric or horospherical.
Main applications are found in an affine sieve on orbits of thin groups as well
as in sphere counting problems for sphere packings invariant under a geometrically finite group.
In our sphere counting problems, spheres can be ordered with respect to a general conformal metric.
\end{abstract}

\maketitle
\tableofcontents

\section{Introduction}
In these notes, we will discuss counting and equidistribution problems for orbits of thin groups
in homogeneous spaces. 

\medskip
Let $G$ be a connected semisimple Lie group and $H$ a closed subgroup. 
We consider the homogeneous space $V=H\ba G$ and fix the identity coset $x_0=[e]$.
Let $\G$ be a discrete subgroup of $G$ such that the orbit $x_0\G$ is discrete, and let $\{B_T:T>1\}$ be
a family of compact subsets of $V$ whose volume tends to infinity as $T\to \infty$.
 Understanding the asymptotic
of $\#(x_0\G\cap B_T)$ is a fundamental problem which bears many applications 
in number theory and geometry. 
We refer to this type of counting problem as an archimedean counting as opposed to a combinatorial counting where the elements
in $x_0\G$ are ordered with respect to a word metric on $\G$; both have been used in applications to sieves (see \cite{Kow} and \cite{Kow1}).


\medskip

When $\G$ is a lattice in $G$, i.e., when $\G\ba G$ admits a finite invariant measure,
this problem is well understood for a large class of subgroups $H$, e.g., when $H$ is a maximal subgroup:
 assuming that the boundaries of $B_T$'s are sufficiently regular and that $H\cap \G$ is a lattice in $H$,
  $\#(x_0\G\cap B_T)$
is asymptotically proportional to the volume of $B_T$, computed  with respect to a suitably normalized $G$-invariant
measure in $V$ (\cite{DRS}, \cite{EM}, \cite{EMS}, etc.).

When $\G$ is a thin group, i.e., a
Zariski dense subgroup which is not a lattice in $G$, it is far less understood in a general setting.
In the present notes, we focus on the case when
 $G$ is the special orthogonal group $\SO(n,1)$ and $H$ is either a symmetric subgroup or a horospherical subgroup
(the case of $H$ being the trivial subgroup will be treated as well).
In this case, we have 
a more or less satisfactory understanding for the counting problem for groups $\G$ equipped with a certain finiteness property,
called the geometric finiteness (see Def. \ref{gfd}).
\medskip

By the fundamental observation of Duke, Rudnick and Sarnak \cite{DRS},
the counting problem for $x_0\G\cap B_T$ can be well approached via the following equidistribution problem:
$$ \text{{\it Describe the asymptotic distribution of $\G\ba \G H g$ in $\G\ba G$ as $g\to \infty$.}} $$

The assumption that $H$ is either symmetric or horospherical
is made so that we can approximate the translate $\G\ba \G H g$ locally 
by the matrix coefficient function in the quasi-regular representation space $L^2(\G\ba G)$.
This idea goes back to Margulis's thesis in 1970 \cite{Ma} (also see \cite{KM}),
but a more systematic formulation in our setting is due to Eskin and McMullen \cite{EM}.

The fact that the trivial representation is contained in $L^2(\G\ba G)$ for $\G$ lattice
is directly related to the phenomenon that, when $\G\ba \G H$ is of finite volume, the translate
$\G\ba \G H g$ becomes equidistributed in $\G\ba G$
 with respect to a $G$-{\it invariant} measure as $g\to \infty$ in $H\ba G$.
When $\G$ is not a lattice, the minimal subrepresentation, that is, the subrepresentation with the 
slowest decay of matrix coefficients
 of $L^2(\G\ba G)$, is infinite dimensional and this makes
the distribution of $\G\ba \G H g$ much more intricate, and
understanding it requires introducing several singular measures in $\G\ba G$.
A key input is the work of Roblin \cite{Ro} on the asymptotic of matrix coefficients
for $L^2(\G\ba G)$, which he proves using ergodic theoretic methods.

\medskip

There is a finer distinction among geometrically finite groups depending on the size of their critical exponents.
When the critical exponent $\delta$ of $\Gamma$ is bigger than $(n-1)/2$,
the work of Lax and Phillips \cite{LP} implies a spectral gap for $L^2(\G\ba \bH^n)$;
which we will call a {\it spherical} spectral gap for $L^2(\Gamma\ba G)$ as it concerns only the spherical part of $L^2(\Gamma\ba G)$.
We formulate a notion of a spectral gap for $L^2(\G \ba G)$ which deals with both spherical and non-spherical parts, based
on the knowledge of the unitary dual of $G$.  
Under the hypothesis that $L^2(\Gamma\ba G)$ admits a spectral gap (this is known to be true if $\delta>n-2$),
developing Harish-Chandra's work on harmonic analysis on $G$ in combination with Roblin's work on ergodic theory,
 we obtain an effective version of the asymptotic of matrix coefficients for $L^2(\G\ba G)$ in joint work with Mohammadi \cite{MO2}.
This enables us to state an effective equidistribution of $\G\ba \G H g$ in $\G\ba G$.
Similar to the condition that $\G\ba \G H$ is of finite volume in the case of $\G$ lattice,
there is also a certain restriction on the size of the orbit $\G\ba \G H$ in order to deduce such an equidistribution.
A precise condition is that the skinning measure $\mu_H^{\PS}$ of $\G\ba \G H$, introduced in \cite{OS}, is finite; roughly speaking $|\mu_H^{\PS}|$ measures 
asymptotically the portion of $\G\ba \G H$ which returns to a compact subset after flowed by the geodesic flow.
When the skinning measure $\mu_H^{\PS}$ is compactly supported, the passage from the asymptotic
of the matrix coefficient to the equidistribution of $\G\ba \G H g$ can be done by the so-called usual
thickening methods. However when  $\mu_H^{\PS}$ is not compactly supported, 
this step requires a genuinely different strategy from the lattice case via the study
of the transversal intersections, carefully done in \cite{OS}.

\medskip
The error term in 
our effective equidistribution result of $\G\ba \G H g$ depends only on the spectral gap data of $\G$. 
This enables us to state
 the asymptotic of $\#(x_0\G_d\gamma \cap B_T)$ effectively in a uniform manner
 for all $\gamma\in \G$ and
for any family $\{\G_d<\G\}$ of subgroups of finite index which has a uniform spectral gap, if they satisfy $\G_d\cap H=\G\cap H$.
 When $\G$ is a subgroup of an arithmetic subgroup of $G$ with $\delta>n-2$,
the work of Salehi-Golsefidy and Varju \cite{SV}, extending an earlier work
of Bourgain, Gamburd and Sarnak \cite{BGS}, 
provides a certain congruence family $\{\G_d\}$ satisfying this condition.

A recent development on an affine sieve  \cite{BGS2}
 then tells us that such a uniform effective counting statement can be used
to describe the distribution of almost prime vectors, as well as to give a sharp upper bound for primes (see Theorem \ref{affine}).

\medskip

One of the most beautiful applications of the study of thin orbital counting problem can
be found in Apollonian circle packings. We will describe this application as well
as its higher dimensional analogues. The ordering in counting circles can be done not only in the Euclidean metric,
 but also in general conformal metrics.
 It is due to this flexibility that we can also describe the asymptotic number of circles
in the ideal triangle of the hyperbolic plane ordered by the hyperbolic area in the last section.

\medskip

\noindent{\bf Acknowledgment} These notes are mostly based on the papers \cite{OS} and \cite{MO2}.
I am grateful to Nimish Shah and Amir Mohammadi for their collaborations.
I would also like to thank Tim Austin for useful discussions concerning the last section
of these notes.

\section{$\G$-invariant conformal densities and measures on $\G\ba G$}
In the whole article, let $(\bH^n, d)$ be the $n$-dimensional real hyperbolic space with constant curvature $-1$
and $\partial(\bH^n)$ its geometric boundary.
Set $G:=\op{Isom}^+(\bH^n)\simeq \SO(n,1)^\circ$. Let $\G$ be a discrete torsion-free subgroup of $G$.
 We assume that $\G$ is non-elementary, or equivalently,
$\G$ has no abelian subgroup of finite index.

We review some basic geometric and measure theoretic
concepts for  $\G$ and define several locally finite Borel measures (i.e., Radon measures) on $\G\ba G$ associated to 
$\G$-invariant conformal densities
on $\partial(\bH^n)$. 
When $\G$ is a lattice, these measures all coincide with each other, being simply
a $G$-invariant measure. But for a thin subgroup $\G$,
they are all different and singular, and appear in our equidistribution and counting statements.
General references for this section are \cite{Ratc}, \cite{Bow}, \cite{Pa}, \cite{Su}, \cite{Sullivan1984}, \cite{Ro}.

\medskip
 
\subsection{Limit set and geometric finiteness}
We denote by $\delta_\G=\delta$ the critical
exponent of $\G$, i.e., the abscissa of convergence of the Poincare series
$\sum_{\gamma\in \G}e^{-s d(o, \gamma(o))}$ for $o\in \mathbb H^n$. We have $0<\delta\le n-1$.
  The limit set $\Lambda(\G)$ is defined to be the set of all accumulation points of $\G(z)$ in 
the compactification $\overline{\bH^n}=\bH^n\cup\partial(\bH^n)$, $z\in \bH^n$.
As $\G$ is discrete, $\Lambda(\G)$ is contained in the boundary $\partial(\bH^n)$.

\begin{dfn}\label{gfd} \rm
\begin{enumerate}
\item The convex core $C(\G)$ of $\G$ is the quotient by $\G$ of the smallest convex subset of $\bH^n$ containing all geodesics
connecting points in $\Lambda(\G)$.
\item $\G$ is called geometrically finite (resp. convex cocompact)
 if the unit neighborhood of the convex core of $\G$ has finite volume (resp. compact).
\end{enumerate}
\end{dfn}

A lattice is clearly a geometrically finite group and so is a discrete group admitting a finite sided convex fundamental
domain in $\bH^n$. An  
 important characterization of a geometrically finite group is given in terms of its limit set. For this, we need
to define:
a point $\xi\in \Lambda(\G)$ is called a parabolic limit point for $\G$ if $\xi$ is a unique fixed
point in $\partial{\bH^n}$ for an element of $\G$ and a radial limit point if the projection of
 a geodesic ray $\xi_t$ toward $\xi$
in $\G\ba \bH^n$ meets a compact subset  for an unbounded sequence of time $t$.

Now $\G$ is geometrically finite if and only if $\Lambda(\G)$ consists only of parabolic
and radial limit points \cite{Bow}. For $\G$ geometrically finite, its critical exponent
 $\delta$ is equal to the Hausdorff dimension
of $\Lambda(\G)$, and is $n-1$ only when $\G$ is a lattice in $G$ \cite{Su}.



\subsection{Conformal densities}
To define a conformal density, we first
recall the Busemann function $\beta_{\xi}(x,y)$ for $x,y\in \bH^n$ and $\xi\in \partial(\bH^n)$:
$$\beta_{\xi}(x,y)=\lim_{t\to \infty} d(x, \xi_t) -d(y, \xi_t)$$
where $\xi_t$ is a geodesic toward $\xi$.
Hence
 $\beta_{\xi}(x,y)$ measures a signed distance between horospheres based at $\xi$ passing through
$x$ and $y$ (a horosphere based at $\xi$ is a Euclidean sphere in $\bH^n$ tangent at $\xi$).
\begin{dfn} \rm A
$\G$-invariant conformal density of dimension $\delta_\mu>0$ is
a family $\{\mu_x: x\in \bH^n\}$ of finite positive measures on  $\partial(\bH^n)$ satisfying:
\begin{enumerate}
 \item For any $\gamma\in \G$, $\gamma_*\mu_x=\mu_{\gamma(x)}$;
\item $\frac{d\mu_x}{d\mu_y}(\xi)=e^{\delta_\mu \beta_{\xi}(y,x)}$ for all $x,y\in \bH^n$ and $\xi\in \partial(\bH^n)$.
\end{enumerate}
\end{dfn}

 It is easy to construct such a density of dimension $n-1$, as 
we simply need to set $m_x$ to be
 the $\op{Stab}_G(x)$-invariant probability measure on $\bH^n$. This is a unique up to scaling,
 and does not depend on $\G$. We call it the Lebesgue density.

How about in other dimensions?
A fundamental work of Patterson \cite{Pa}, generalized by Sullivan \cite{Su}, shows the following by an explicit construction:

\begin{thm}
 There exists a $\G$-invariant conformal density of dimension $\delta_\G=\delta$.
\end{thm}

Assuming that  $\G$ is of divergence type,
i.e., its Poincare series diverges at $s=\delta$, Patterson's construction can be summarized as follows:
 Fixing  $o\in \bH^n$, for each $x\in \bH^n$, consider the finite measure on $\overline{\bH^n}$
given by
$$\nu_{x,s}=\frac{1}{\sum_{\gamma\in \G} e^{-sd(o, \gamma (o))}} \sum_{\gamma\in \G}  e^{-sd(x, \gamma (o))}\delta_{\gamma(o)} .$$ 
Then $\nu_x$ is the (unique) weak-limit of $\nu_{x,s}$ as $s\to \delta^+$, and
$\{\nu_x:x \in \bH^n\}$ is the desired density of dimension $\delta$.

\medskip

In the following, we fix a $\G$-invariant conformal density $\{\nu_x\}$
 of dimension $\delta$, and call it the Patterson-Sullivan density (or simply the PS density). 
It is known to be unique up to scaling, when $\G$ is of divergence type, e.g., geometrically finite groups.

Denoting by $\Delta$ the hyperbolic Laplacian on $\bH^n$,
the PS density is closely related to the bottom of the spectrum of $\Delta$ for its action on smooth
functions on $\G\ba \bH^n$.
If we set $\phi_0(x):=|\nu_x|$ for each $x\in \bH^n$,
then $$\phi_0(x)=\int_{\xi\in \Lambda(\G)} d\nu_x(\xi)=\int_{\xi\in \Lambda(\G)} \frac{d\nu_x}{d\nu_o}(\xi) \; d\nu_o(\xi)
=\int_{\xi\in \Lambda(\G)} e^{\delta \beta_{\xi}(o,x)} d\nu_o(\xi) .$$
Since $\gamma_*\nu_x=\nu_{\gamma(x)}$,
we note that $\phi_0(\gamma(x))=\phi_0(x)$, i.e., $\phi_0$ is a function on $\G\ba \bH^n$, which is
{\it positive} everywhere!
Furthermore we have
\begin{enumerate}
\item $\Delta(\phi_0)=\delta (n-1-\delta)\phi_0$;
 \item if $\G$ is geometrically finite, then $\phi_0\in L^2(\G\ba \bH^n)$ if and only if $\delta >(n-1)/2$.
\end{enumerate}

\subsection{Measures on $\G\ba G$ associated to a pair
of conformal densities}
For $v\in \T^1(\bH^n)$, we denote by $v^{+}$ and $v^-$ the forward and backward endpoints of the geodesic determined by $v$.
 Fixing $o\in \bH^n$,
the map $v\mapsto (v^+, v^-, s=\beta_{v^-}(o, v))$
 gives a homeomorphism between $\T^1(\bH^n)$ and $(\partial(\bH^n)\times \partial(\bH^n) -\text{diagonal}) \times \br$.
Therefore we may use the coordinates $(v^+, v^-, s=\beta_{v^-}(o, v))$ of $v$
in order to define measures on $\T^1(\bH^n)$.

\medskip
Let $\{\mu_x\}$ and $\{\mu_x'\}$ be $\G$-invariant conformal densities on
$\partial{(\bH^n)}$ of dimensions $\delta_\mu$ and $\delta_{\mu'}$ respectively.
 After Roblin~\cite{Ro}, we define a measure $\tilde m^{\mu,\mu'}$ on
 $\T^1(\bH^n)$ associated to $\{\mu_x\}$ and $\{\mu_x'\}$ by
\begin{equation} \label{eq:m-mumu}
d \tilde m^{\mu,\mu'}(v) =
e^{\delta_\mu \beta_{v^+}(o, v)}\; e^{\delta_{\mu'} \beta_{v^-}(o,v) }\;d\mu_o(v^+) d
\mu'_o(v^-) ds.
\end{equation}

It follows from the $\G$-invariant conformal properties of $\{\mu_x\}$ and $\{\mu_x'\}$ that
 the definition of $\tilde m^{\mu,\mu'}$ is independent of the choice of $o\in \bH^n$ and that
$\tilde m^{\mu,\mu'}$ is left $\G$-invariant. Hence
it induces a Radon measure $m^{\mu,\mu'}$ on the quotient space $\T^1(\G\ba \bH^n)=\G\ba \T^1(\bH^n)$.

We will lift the measure $m^{\mu,\mu'}$ to $\G\ba G$. This lift depends on the choice of
subgroups $K$, $M$ and $A=\{a_t\}$ of $G$. Here 
 $K$ is a maximal compact subgroup of $G$ and $M$ is
the stabilizer of a vector $X_0\in \T^1(\bH^n)$ based at $o\in \bH^n$ with $K=G_o$. 
Via the isometric action of $G$, 
we may identify the quotient spaces $G/K$ and $G/M$ with  $\bH^n$ and $\T^1(\bH^n)$ respectively.
Let $A=\{a_t\}$ be the one parameter subgroup
of diagonalizable elements of $G$ such that
the right multiplication by $a_t$ on $G/M$ corresponds to the geodesic flow on $\T^1(\bH^n)$ for time $t$.

By abuse of notation, we use the same notation
$m^{\mu,\mu'}$ for the $M$-invariant extension of $m^{\mu,\mu'}$ on $\G\ba G/M=\G\ba \T^1(\bH^n)$ to $\G\ba G$, that is,
for $\Psi\in C_c(\G\ba G)$,
$$m^{\mu,\mu'}(\Psi)=\int_{x\in \G\ba G/M} \Psi^M(x) dm^{\mu,\mu'}(x)$$
where $\Psi^M(x)=\int_M \Psi(xm) dm $ for the probability $M$-invariant measure $dm$ on $M$.

The measures $m^{\mu,\mu'}$ on $\G\ba G$ where $\mu$ and $\mu'$ are 
 the PS-density $\{\nu_x\}$ or the Lebesgue density $\{m_x\}$
are of special importance.
We name them as follows:
\begin{itemize}
 \item{Bowen-Margulis-Sullivan measure:} $m^{\BMS}:=m^{\nu,\nu}$;
\item{Burger-Roblin measure:} $m^{\BR}:=m^{m, \nu}$;
\item{Burger-Roblin $*$-measure:} $m^{\BR}_*:=m^{\nu,m}$;
\item{Haar measure:} $m^{\Haar}:=m^{m,m}$.
\end{itemize}

For brevity, we refer to these as BMS, BR, BR$_*$, Haar measures respectively.
As the naming indicates, $m^{\Haar}$ turns out to be a $G$-invariant measure.
For $g\in G$, we use the notation $g^{\pm}$ for $(gM)^{\pm}$
with $gM$ considered as a vector in $\T^1(\bH^n)$.
It is clear from the definition that the supports of BMS, BR, BR$_*$ measures 
are respectively given by $\{g\in \G\ba G: g^{\pm}\in \Lambda(\G)\}$,
$\{g\in \G\ba G: g^{-}\in \Lambda(\G)\}$, and
$\{g\in \G\ba G: g^{+}\in \Lambda(\G)\}$. In particular, the support of BMS measure is contained in the convex core
of $\Gamma$. Sullivan showed that $|m^{\BMS}|<\infty$ if $\G$ is geometrically finite.

The BMS, BR and BR$_*$ measures are respectively invariant under $A$, $N^+$ and $N^-$
where $N^{+}$ (resp $N^-$) denotes the expanding (resp. contracting) horospherical subgroup of $G$ for $a_t$:
$$N^{\pm}=\{g\in G: a_tga_{-t}\to e\text{ as $t\to {\pm \infty}$}\}.$$

The finiteness of $m^{\BMS}$ turns out to be a critical condition for the ergodic theory on $\G\ba G$.

\begin{thm}\label{im} Suppose that $|m^{\BMS}|<\infty$ and that $\G$ is Zariski dense.
 
\begin{enumerate}
 \item $m^{\BMS}$ is $A$-mixing:
for any $\Psi_1, \Psi_2\in L^2(\G\ba G)$,
$$\lim_{t\to\infty}\int_{\G\ba G} \Psi_1(ga_t)\Psi_2(g)\; dm^{\BMS}(g)=\frac{1}{|m^{\BMS}|} m^{\BMS}(\Psi_1)\cdot m^{\BMS}(\Psi_2) .$$
\item Any locally finite $N^+$-ergodic invariant measure on $\Gamma\ba G$
is either supported on a closed $N^+M_0$-orbit where $M_0$ is an abelian closed subgroup of $M$
or $m^{\BR}$.

\item $m^{\BR}$ is a finite measure if and only if $\G$ is a lattice in $G$ 
\end{enumerate}
\end{thm}

Claim (1) was first made in Flaminio and Spatzier \cite{FS}. However there is a small gap in their proof
which is now fixed by Winter \cite{Wi}. Winter also obtained Claim (1) in a general rank one symmetric space.
  For M-invariant functions,
this claim was earlier proved by Babillot \cite{Ba} and in this case the Zariski density
assumption is not needed.  Claim (2) was first proved by
Burger \cite{Bu} for a convex cocompact surface with critical exponent bigger than $1/2$.
In a general case, Winter obtained Claim (2) from Roblin's work \cite{Ro} and Claim (1).

Claim (3) is proved in \cite{OS}, using
Ratner's measure classification [1991] of finite measures invariant under unipotent flows. Namely, we show that m
BR is not one of those homogeneous measures
that her classification theorem lists for finite invariant measures.

In the spirit of Ratner's measure classification theorem \cite{Ra}, we pose the following question:
\begin{Q} Under the assumption of Theorem \ref{im}, let $U$ be a connected unipotent subgroup 
of $G$ or more generally a connected subgroup generated by unipotent one-parameter subgroups.
\begin{enumerate}
 \item 
 Classify all locally finite $U$-invariant ergodic measures in $\G\ba G$.
\item Describe the closures of $U$-orbits in $\G\ba G$.
\end{enumerate}
\end{Q}
The emphasis here is that we want to understand not only finite measures but all Radon measures. For $U$ horospherical group, one should be able to
answer these questions by the methods of Roblin \cite{Ro}.
In general, this seems to be a very challenging question. 
We mention a recent related result of \cite{MO}: if $\G$ is a convex cocompact subgroup of
$G$ and $U$ is a connected unipotent subgroup of $N^+$ of dimension $k$,
$m^{\BR}$ is $U$-ergodic if $\delta>(n-1)-k$. 
Precisely speaking, this is proved only for $n=3$ in \cite{MO}, but the methods of proof works
for a general $n\ge 3$ as well.

\section{Matrix coefficients for $L^2(\G\ba G)$}
Let $\G$ be a discrete, torsion-free, non-elementary subgroup of $G=\SO(n,1)^\circ$. 
The right translation action of $G$ on $L^2(\G\ba G, m^{\Haar})$ gives rise to a unitary representation,
as $m^{\Haar}$ is $G$-invariant.
For $\Psi_1, \Psi_2\in L^2(\G\ba G)$,
the matrix coefficient function 
is a smooth function on $G$ defined by $$g\mapsto \la g. \Psi_1, \Psi_2\ra:=\int_{\G\ba G} \Psi_1(xg) \overline{\Psi_2(x)} dm^{\Haar}(x) .$$

Understanding the asymptotic expansion of 
$\la a_t. \Psi_1, \Psi_2\ra $ (as $t\to \infty$) is a basic problem in harmonic analysis
 as well as a main tool in our approach to the counting problem.

The quality of the error term in this type of the asymptotic expansion usually depends on Sobolev norms of $\Psi_i$'s.
For $\Psi\in C^\infty(\G\ba G)$ and $d\in \N$, 
the $d$-th Sobolev norm of $\Psi$ is given by $\mathcal S_{ d}(\Psi)=\sum \| X (\Psi)\|_{2}$
where the sum is taken over all monomials $X$ in some fixed basis of the Lie algebra of $G$
of order at most $d$
and $\|X(\Psi)\|_2$ denotes the $L^2$-norm of $\Psi$.

\medskip
In order to describe our results as well as a conjecture, we begin by describing the unitary dual of $G$, i.e.,
the set of equivalence classes of all irreducible unitary representations of $G$.

\medskip 
\noindent{\bf  The unitary dual $\hat G$.} 
Let $K$ be a maximal compact subgroup of $G$
and fix a Cartan decomposition $G=KA^+K$.
We parametrize $A^+=\{a_t: t\ge 0\}$ so that $a_t$ corresponds to
the geodesic flow on $\T^1(\bH^n)=G/M$ where $M:=C_K(A^+)$ is the centralizer of $A^+$ in $K$.

A representation $\pi\in \hat G$ is said to be {\it tempered} if for any $K$-finite vectors $v_1, v_2$
of $\pi$, the matrix coefficient function $g\mapsto \la \pi(g)v_1, v_2\ra$ belongs to $L^{2+\e}(G)$ for any $\e>0$.
We write $\hat G=\hat G_{\op{temp}}\cup \hat G_{\text{non-temp}}$ as the disjoint union of
tempered representations and non-tempered representations.

The work of Hirai \cite{Hi} on the classification of $\hat G$ implies that non-tempered part
of the unitary dual $\hat G$ consists of the trivial representation, and complementary series representations $\mathcal U(\upsilon, s-n+1)$
parameterized by $\upsilon\in \hat M$ and $s\in I_\upsilon$, 
where $\hat M$ is the unitary dual of $M$ and $I_\upsilon$ is an interval contained in $((n-1)/2, n-1)$, depending on $\upsilon$
(see also \cite[Prop. 49, 50]{KS}).
Moreover $\mathcal U(\upsilon, s-n+1)$ is spherical if and only if $\upsilon$ is the trivial representation $1$ of $M$.
By choosing a  Casimir operator 
$\mathcal C$ of the Lie algebra of $G$ normalized so that it acts
 on $C^\infty(G)^K=C^\infty(\bH^n)$ by the negative Laplacian,
the normalization is made so that
 $\mathcal C$ acts on $\mathcal U(1, s-n+1)^\infty$ by the scalar $s(s-n+1)$.


\subsection{Lattice case: $|m^{\Haar}|<\infty$}
Set $$L_0^2(\G\ba G):=\{\Psi\in L^2(\G\ba G): \int \Psi \; dm^{\Haar}=0\}.$$
When $\G$ is a lattice, we have $L^2(\G\ba G)=\c \oplus L_0^2(\G\ba G)$.
 It is well-known that there exists $(n-1)/2<s_0<(n-1)$ such that
 $L_0^2(\G\ba G)$ does not contain
any complementary series representation of parameter $s\ge s_0$ \cite{BG}.
This implies the following (cf. \cite[Prop 5.3]{KO}):

\begin{thm} Suppose $\G$ is a lattice in $G$, that is,
$|m^{\Haar}|<\infty$. There exists $\ell \in \N$ such that for any $\Psi_1, \Psi_2 \in  L^2(\G\ba G) \cap
C(\G\ba G)^\infty$, we have,
as $t\to \infty$,
$$\la a_t. \Psi_1, \Psi_2\ra  =\frac{1}{|m^{\Haar}|} m^{\Haar}(\Psi_1) m^{\Haar}(\Psi_2) +
 O(\mathcal S_\ell(\Psi_1) \mathcal S_\ell(\Psi_2)
e^{(s_0-n+1)t})$$
where $m^{\Haar}(\Psi_i)=\int_{\G\ba G} \Psi_i(x) dm^{\Haar}(x)$ for $i=1,2$.
\end{thm}

We note that the constant function $1/\sqrt{|m^{\Haar}|}$ is a unit vector
in $L^2(\G\ba G)$ and
the main term $\frac{1}{|m^{\Haar}|} m^{\Haar}(\Psi_1) m^{\Haar}(\Psi_2)$
is simply the product of the projections of $\Psi_1$ and $\Psi_2$
to the minimal subrepresentation space, which is $\c$, of $L^2(\G\ba G)$.

\subsection{Discrete groups with $|m^{\BMS}|<\infty$}
When $\G$ is not a lattice, we have $L^2(\G\ba G)=L^2_0(\G\ba G)$.
By the well-known decay  of the matrix coefficients
of unitary representations with no $G$-invariant vectors due to Howe and Moore \cite{HM},
 we have,
for any $\Psi_1, \Psi_2\in L^2(\G\ba G)$,
$$\lim_{t\to\infty} \la a_t. \Psi_1, \Psi_2\ra = 0.$$ 
However we have a much more precise description on the decay of $\la a_t. \Psi_1, \Psi_2\ra$ due to Roblin:

\begin{thm}\cite{Ro}\label{lo} Let $\G$ be Zariski dense with $|m^{\BMS}|<\infty$.
For any $\Psi_1,\Psi_2\in C_c(\G\ba G)$,
$$\lim_{t\to \infty}
e^{(n-1 -\delta)t} \la a_t. \Psi_1, \Psi_2\ra = \frac{m^{\BR}(\Psi_1)\cdot m^{\BR}_*(\Psi_2)}{|m^{\BMS}|} .$$
\end{thm}

Roblin proved this theorem for $M$-invariant functions using the mixing of the geodesic flow
 due to Babillot \cite{Ba}. His proof extends without difficulty to general functions,
 based on the $A$-mixing in $\G\ba G$ stated as in Theorem \ref{im}.

\subsection{Geometrically finite groups with $\delta>(n-1)/2$}
In this subsection, we assume that $\G$ is  a geometrically finite,
Zariski dense, discrete subgroup of $G$ with $\delta>(n-1)/2$.
Under this assumption, the works of Lax and Phillips \cite{LP} and Sullivan \cite{Sullivan1984} together imply that
there exist only finitely many $(n-1)\ge s_0>s_1\ge \cdots\ge s_\ell >(n-1)/2$ such that the spherical complementary series representation $\mathcal U(1, s-n+1)$
occurs as a subrepresentation of $L^2(\G\ba G)$ and $s_0=\delta$.
In particular, there is no spherical complementary
representation $\mathcal U(1, s-n+1)$ contained in $L^2(\G\ba G)$
for $\delta<s<s_1$; hence we have a {\it spherical} spectral gap for $L^2(\G\ba G)$.

 Using the classification of $\hat G_{\op{non-temp}}$,
 we formulate the notion of a spectral gap.
Recall that a unitary representation $\pi$ is said to be weakly contained in a unitary representation
$\pi'$  if any diagonal matrix coefficients of $\pi$ can be approximated, uniformly on compact subsets, by
convex combinations of diagonal matrix coefficients of $\pi'$.


\begin{dfn}\label{intro_strong_gap}\label{sng}  \rm  We say that $L^2(\G\ba G)$ has a {\it strong spectral gap} if
   \begin{enumerate}
    \item $L^2(\G\ba G)$ does not contain any $\mathcal U(\upsilon,\delta-n+1)$ with $\upsilon\ne 1$;
\item there exist
$\tfrac{n-1}{2}<s_0(\G)<\delta$ such that  $L^2(\G\ba G)$ does not weakly contain any $\mathcal U(\upsilon,s-n+1)$ with $s\in (s_0(\G), \delta)$ and $\upsilon\in \hat M$. \end{enumerate}
\end{dfn}

For $\delta\le \tfrac{n-1}2$, the Laplacian spectrum of $L^2(\G\ba \bH^n)$ is continuous; this implies that
there is no spectral gap for $L^2(\G\ba G)$.

\begin{Con}[Spectral gap conjecture]\cite{MO2}\label{conj}
 If $\G$ is a geometrically finite and Zariski dense subgroup of $G$ with $\delta>\tfrac{n-1}2$,  $L^2(\G\ba G)$ has a strong spectral gap.
\end{Con}

If $\delta>(n-1)/2$ for $n=2,3$, or if $\delta>(n-2)$ for $n\ge 4$, then
$L^2(\G\ba G)$ has a strong spectral gap.
This observation follows from the classification of $\hat G_{\op{non-temp}}$ which says that
there is no non-spherical complementary series representation $\mathcal U(\upsilon, s-n+1)$ of parameter $n-2<s<n-1$ \cite{Hi}.

Our theorems are proved under the following slightly weaker spectral gap property assumption:
\begin{dfn}\cite{MO2}\label{introsg}\label{sngg}  \rm  We say that $L^2(\G\ba G)$ has a {\it spectral gap} if there exist
$\tfrac{n-1}{2}<s_0=s_0(\G)<\delta$ and $n_0=n_0(\G)\in \N$ such that 
   \begin{enumerate}
    \item the multiplicity of $\mathcal U(\upsilon,\delta-n+1)$ contained in $L^2(\G\ba G)$ is at most $\dim(\upsilon)^{n_0}$ for any $\upsilon\in \hat M$;
\item  $L^2(\G\ba G)$ does not weakly contain any  $\mathcal U(\upsilon,s-n+1)$ with $s\in (s_0, \delta)$ and  $\upsilon\in \hat M$.\end{enumerate}
  The pair $(s_0(\G), n_0(\G))$ will be referred to as the spectral gap data for $\G$.
\end{dfn}

The spectral gap hypothesis implies that for $\Psi_1,\Psi_2 \in L^2(\G\ba G)$, the leading term of the
asymptotic expansion of the matrix coefficient $\la a_t .\Psi_1, \Psi_2\ra $ 
is determined  by  $\la a_t. P_{\delta} (\Psi_1), P_{\delta} (\Psi_2)\ra $ where $P_{\delta}$ is the projection operator
from $L^2(\G\ba G)$ to $\mathcal H_{\delta}^\dag$, which is the sum of all complementary
series representations $\mathcal U(\upsilon, \delta-n+1)$, $\upsilon\in \hat M$ occurring
in $L^2(\G\ba G)$ as sub-representations.

Building up on the work of Harish-Chandra on the asymptotic
behavior of the Eisenstein integrals (cf. \cite{Wa}, \cite{Wa2}), we obtain an asymptotic formula for 
$\la a_t v, w\ra $ for all $K$-isotypic vectors $v,w\in \mathcal H_{\delta}^\dag$.
This extension alone does not quite explain the leading term
of $\la a_t P_{\delta} (\Psi_1), P_{\delta} (\Psi_2)\ra $ in terms of functions $\Psi_1$ and $\Psi_2$; 
however, with the help of Theorem \ref{lo},  we are able to
 prove the following:
\begin{thm}\cite{MO2}\label{harmixing} Suppose that $L^2(\G\ba G)$ possesses a spectral gap. 
  There exist $\eta_0>0$ and  $\ell\in \N$ 
  such that
for any $\Psi_1, \Psi_2\in C_c^\infty(\G\ba G)$, as $t\to \infty$,
$$ e^{(n-1-\delta)t} \la a_t \Psi_1, \Psi_2\ra 
 = \frac{m^{\BR}(\Psi_1)\cdot m^{\BR}_*(\Psi_2)}{|m^{\BMS}|}
+O(\mathcal S_\ell(\Psi_1) \mathcal S_\ell(\Psi_2) e^{-\eta_0 t} ).$$
\end{thm}

We reiterate that the leading term in Theorem \ref{harmixing} is from the ergodic theory and the
error term is from the harmonic analysis and the spectral gap. 

\subsection{Effective mixing for $m^{\BMS}$} 
Theorem \ref{harmixing} can be used to obtain an effective mixing
for the BMS measure (i.e., an effective version of Theorem \ref{im}(1)) for
geometrically finite, Zariski dense subgroups with a spectral gap \cite{MO2}.

For geometrically finite groups with $\delta\le (n-1)/2$,
 we cannot expect Theorem \ref{harmixing} to hold for such groups.
However a recent work of Guillarmou and Mazzeo \cite{GM} establishes 
meromorphic extensions of the resolvents of the Laplacian and of the Poincare series. 

For $\G$ convex cocompact,
 Stoyanov obtained, via the spectral properties of Ruelle transfer operators,
an effective mixing of the geodesic flow for the BMS measure, regardless of the size
of the critical exponent (see \cite{St}). We remark that the idea of using spectral estimates of
Ruelle transfer operators in obtaining an effective mixing
was originated in Dolgopyat's work \cite{Do}.

It will be interesting to
see if the results in \cite{GM} can be used to answer the question:
\begin{Q}
 Prove an effective mixing  of the  geodesic flow for the BMS measure for all geometrically
finite groups.
\end{Q}

\medskip
We close this section by posing the following:
\begin{Q}
Find a suitable analogue of Theorem \ref{harmixing} for a higher rank simple Lie group $G$ such
as $\SL_2(\br)\times \SL_2(\br)$ or $\SL_3(\br)$.
\end{Q}

\section{Distribution of $\G$ in $G$}\label{cg}
For a family  $\{B_T: T>1\}$ of compact subsets in $G$,
 the study of the asymptotic of $\#(\G\cap B_T)$ can be approached directly
by Theorems \ref{lo} and \ref{harmixing}.
The link is given by the following function on $\G\ba G\times \G\ba G$:
$$F_T(g,h):=\sum_{\gamma\in \G}\chi_{B_T}(g^{-1}\gamma h) $$
where $\chi_{B_T}$ denotes the characteristic function of $B_T$. 
Observing that $F_T(e,e)=\# (\G\cap B_T)$, we will explain how the asymptotic behavior of $F_T(e,e)$
is related to the matrix coefficient function for $L^2(\G\ba G)$.
For $g=k_1a_tk_2\in KA^+K$, we have
$dm^{\Haar}(g)= \xi(t) dt dk_1 dk_2 $ where $\xi(t)= e^{(n-1)t}(1+O(e^{-\beta t}))$ for some $\beta>0$.

\medskip

Now assume that $\G$ admits a spectral gap, so that
Theorem \ref{harmixing} holds.
 For any real-valued $\Psi_1,\Psi_2\in C_c^\infty(\G\ba G)$, if we set
 $\Psi_i^k(g):=\Psi_i(gk)$, then we can deduce from Theorem \ref{harmixing} that 
\begin{align}\label{ft} & \la F_T, \Psi_1\otimes \Psi_2\ra_{\G\ba G\times \G\ba G} \notag \\&
=\int_{g\in B_T} \la \Psi_1, g.\Psi_2\ra_{L^2(\G\ba G)} \;  dm^{\Haar}(g) \notag \\
&= \int_{k_1a_tk_2\in B_T} \la \Psi_1^{k_1^{-1}}, a_t. \Psi_2^{k_2} \ra \cdot \xi(t) dt dk_1 dk_2 \notag \\
&=  \tfrac{1}{|m^{\BMS}|}\int_{k_1a_tk_2\in B_T} \left( m^{\BR}_*(\Psi_1^{k_1^{-1}})m^{\BR}( \Psi_2^{k_2}) e^{\delta t} 
+ O(e^{(\delta -\eta)t}) \right) dt dk_1 dk_2   
\end{align}
for some $\eta>0$, with the implied constant depending only on the Sobolev norms of $\Psi_i$'s.

Therefore,
if $F_T(e,e)$ can be {\it effectively approximated} by
$\la F_T, \Psi_\e\otimes \Psi_\e\ra $ for an approximation of identity $\Psi_\e$ in $\G\ba G$,
a condition which depends on the regularity of the boundary of
$B_T$ relative to the Patterson-Sullivan density,
then $F_T(e,e)$ can be computed by evaluating the integral \eqref{ft} 
for $\Psi_1=\Psi_2=\Psi_\e$ and by taking $\e$ to be a suitable power of $e^{-t}$.

To state our counting theorem more precisely, we need:
\begin{dfn}\label{group} 
Define a Borel measure $\mathcal M_{G}$ on $G$ as follows: for $\psi\in C_c(G)$,
\begin{equation*}\label{mdeg}
\mathcal M_{G} (\psi)=\tfrac{ 1}{|m^{\BMS}|} \int_{k_1a_tk_2\in K A^+K} \psi(k_1a_tk_2)
  e^{\delta t} d\nu_o(k_1)dt d\nu_o(k_2^{-1})\end{equation*}
here $\nu_o$ is the $M$-invariant lift to $K$ of the PS measure on $\partial(\bH^n)=K/M$
viewed from $o\in \bH^n$ with $K=\op{Stab}_G(o)$.
\end{dfn}


\begin{dfn}\label{adm1} \rm 
  For a family $\{\B_T\subset  G\}$ of compact subsets
 with $\mathcal M_{G}(\B_T)$ tending to infinity as $T\to \infty$, we say that $\{\B_T\}$
is {\it effectively well-rounded with respect to $\G$}
if there exists $p>0$ such that 
for all small $\e>0$ and $T\gg 1$:
$$\mathcal M_{ G} (B_{T,\e}^+ - B_{T,\e}^-) = O(\e^p \cdot \mathcal M_{G} ( B_T) )$$  
where $B_{T,\e}^+=G_\e B_T G_\e$ and $B_{T,\e}^-=\cap_{g_1, g_2\in G_\e} g_1 B_T g_2$.
Here $G_\e$ denotes a symmetric $\e$-neighborhood of $e$ in $G$.
\end{dfn}

For the next two theorems, we assume that $\G'<\G$ is a subgroup of finite index and both $\G$ and $\G'$ have
spectral gaps.
\begin{thm}\cite{MO2}\label{mo2}
If $\{B_T\}$ is effectively well-rounded with respect to $\G$,  
 then there exists $\eta_0>0$ such that for any $\gamma\in \G$,
$$\# (\G'\gamma \cap B_T )= 
\frac{1}{[\G:\G']}\mathcal M_{G}(B_T)  + O( \mathcal M_{G}(\B_T)^{1-\eta_0})$$
where $\eta_0>0$ depends only on a uniform spectral gap for $\G$ and $\G'$, and the implied
constant is independent of $\G'$ and $\gamma$.
\end{thm}

If $\|\cdot \|$ is a norm on the space $M_{n+1}(\br)$
of $(n+1)\times (n+1)$ matrices, then the norm ball $B_T=\{g\in \SO(n,1):\|g\|\le T\}$ is
effectively well-rounded with respect to $\G$, and hence Theorem \ref{mo2} applies.

Consider the set
$B_T=\Omega_1 A_T \Omega_2$ where $\Omega_i\subset K$ and $A_T=\{a_t: 0\le t\le T\}$. Then
 $\{B_T\}$ is effectively well-rounded with respect to $\G$ if
 there exists $\beta'>0$ such that the PS-measures of the $\e$-neighborhoods of the boundaries of $\Omega_1 M/M$ and $\Omega_2^{-1}M/M$
 are at most  of order $\e^{\beta'}$ for all small $\e>0$. These conditions on $\Omega_i$ are satisfied for instance
if their boundaries in $K/M$ are disjoint from the limit set $\Lambda(\G)$. But also many (but not all) compact subsets
with piecewise smooth boundary also satisfy this condition (see \cite[sect. 7]{MO2}).

Hence
one can deduce the following:
\begin{thm}  \label{secor} \cite{MO2} If $\{B_T=\Omega_1 A_T \Omega_2\}$ is effectively well-rounded with respect to $\G$, then,
as $T\to \infty$,
\begin{equation*}
\#(\G' \gamma \cap \Omega_1A_T\Omega_2) =\frac{\Xi(\G,\Omega_1,\Omega_2)}{ [\G:\G']}{ e^{\delta T}}
 +O(e^{(\delta -\eta_0)T}) \end{equation*}
where  $\Xi(\G,\Omega_1,\Omega_2):=\tfrac{\nu_o(\Omega_1)\nu_o(\Omega_2^{-1})}{\delta \cdot |m^{\BMS}|}
$ and the implied constant is independent of $\G'$ and $\gamma\in \G$.
\end{thm}

The fact that the only dependence of $\G'$ on the right hand side of the above formula
is on the index $[\G:\G']$ is of crucial importance for our intended applications to an affine sieve
(cf. subsection \ref{affine2}).
 Bourgain, Kontorovich and Sarnak \cite{BKS} showed Theorem \ref{secor} for the case $G=\SO(2,1)$
via an explicit computation of matrix coefficients using the Gauss hypergeometric functions,
and Vinogradov \cite{V} generalized their method to $G=\SO(3,1)$.
We note that in view of Theorem \ref{lo}, the non-effective versions of Theorems \ref{mo2} and \ref{secor}  hold
for any Zariski dense  $\G$ with $|m^{\BMS}|<\infty $; this
was obtained in \cite{Ro} (see also \cite{OS}) with a different proof. 
For $\G$ lattice, Theorem 4.5 is known in much greater generality \cite{GO}.




\section{Asymptotic distribution of $\G\ba \G H a_t$}
The counting problem in $H\ba G$ for $H$ a non-trivial subgroup can be approached
via studying the asymptotic distribution of translates $\G\ba \G Ha_t$ as mentioned in the introduction.
We will assume in this section that
 $H$ is either a symmetric subgroup or a horospherical subgroup of $G$. Recall that $H$ is symmetric means that $H$
 is the subgroup of fixed points under an involution of $G$.
In this case, we can relate the distribution of  $\G\ba \G Ha_t$ to
the matrix coefficient function of $L^2(\G\ba G)$.

We fix a generalized Cartan decomposition $G=HAK$ (for $H$ horospherical,
it is just an Iwasawa decomposition) where $K$ is a maximal compact subgroup
and $A$ is a one-dimensional subgroup of diagonalizable elements.
As before, we parametrize $A=\{a_t: t\in \br \}$ so that the right multiplication by $a_t$ on $G/M=\T^1(\bH^n)$
corresponds to the geodesic flow for time $t$.

 When $H$ is horospherical,  we will assume that $H=N^+$, i.e, the expanding horospherical subgroup with respect to $a_t$.


Any symmetric subgroup $H$ of $G$ is known to be locally isomorphic to $\SO(k, 1)\times \SO(n-k)$ for some $0\le k\le n-1$, and 
the $H$-orbit of the identity coset in $G/K$ (resp. $G/M$) is (resp. the unit normal bundle to)
a complete totally geodesic subspace $\bH^k$ of
$\bH^n$ of dimension $k$.
The right multiplication by $a_t$ on $G/M$ corresponds to the geodesic flow
and the image of $Ha_t$ in $G/M$ represents the expansion of the totally geodesic
subspace of dimension $k$ by distance $t$.

\medskip

\subsection{Measures on $\G\ba \G H$ associated to a conformal density}
The leading term in the description of the asymptotic distribution of $\G\ba \G H a_t$ turns out to be
 a new measure on $\G\ba \G H$ associated to the PS density (see \cite{OS}).
We assume that $\G\ba \G H$ is closed in $\G\ba G$ in the rest of this section.

For a $\G$-invariant conformal density
$\{\mu_x\}$ on
$\partial{(\bH^n)}$ of dimension $\delta_\mu$,
  define a measure $\tilde \mu_{H}$ on
 $H/(H\cap M)$  by
\begin{equation*} 
d \tilde\mu_{H}([h]) =
e^{\delta_\mu \beta_{[h]^+}(o, [h])}\; \;d\mu_o([h]^+) 
\end{equation*}
where $o\in \bH^n$ and $[h]\in H/(H\cap M)$ is considered as an element of $G/M=\T^1(\bH^n)$ under
the injective map $H/(H\cap M)\to G/M$.

By abuse of notation, we use the same notation $\tilde \mu_{H}$ for the  $H\cap M$-invariant lift of $\tilde \mu_{H}$ to $H$:
for $\psi\in C_c(\G\ba G)$,
$$\tilde \mu_H(\psi)=\int_{H/(H\cap M)} \psi^{H\cap M} (x) \; d\tilde \mu_{H} (x)$$
where $\psi^{H\cap M} (x)=\int_{H\cap M}\psi(xm) d_{H\cap M}(m)$ for the $H\cap M$-invariant probability measure
$d_{H\cap M}$.
This definition is independent of the choice of $o\in \bH^n$ and
the measure $\tilde \mu_H$ is $H\cap \G$-invariant from the left, and hence
induces a measure $\mu_H$ on $(H\cap \G)\ba H$ or equivalently on $\G \ba \G H$.

For the PS density $\{\nu_x\}$ and the Lebesgue density $\{m_x\}$, the following
two locally finite measures on  $\G \ba \G H$ are of special importance:
\begin{itemize}
\item{Skinning measure:} $\mu^{\PS}_H=\nu_H$;
 \item{$H$-invariant measure:} $\mu^{\Leb}_H=m_H$.
\end{itemize}

We remark that $\mu^{\PS}_H$ is different from $\mu^{\Leb}_H$ in general even when
$H\cap \G$ is a lattice in $H$.

\subsection{Finiteness of the skinning measure $\mu^{\PS}_H$}
The finiteness of the skinning measure $\mu^{\PS}_H$ turns out to be
the precise replacement for the finiteness of the volume measure $\mu^{\Leb}_H$, in extending
the equidistribution statement from $\G$ lattices to thin subgroups.

When is the skinning measure $\mu_H^{\PS}$ finite? This question is completely answered in \cite{OS}
for $\G$ geometrically finite. First, when $H$ is a horospherical subgroup, the support of $\mu_H^{\PS}$ 
is compact.
When $H$ is a symmetric subgroup, i.e., isomorphic to $\SO(k,1)\times \SO(n-k)$ locally,
the answer to this question depends on the notion of the parabolic co-rank of $\G\cap H$: let $\Lambda_p(\G)$ denote the set
of all parabolic limit points of $\G$.
For $\xi\in \Lambda_p(\G)$,  the stabilizer $\G_\xi$ has a free abelian subgroup of finite index, whose rank
is defined to be the rank of ${\xi}$ (or the rank of $\G_{\xi}$).

\begin{dfn}\rm  The parabolic corank of $\G\cap H$ in $\G$ is defined to be the maximum of 
the difference $\text{rank}(\G_\xi)-\text{rank}(\G\cap H)_\xi$ over
all $\xi\in\Lambda_p(\G)\cap \partial(\bH^k).$
\end{dfn}

\begin{thm}\cite{OS} Let $\G$ be geometrically finite.
\begin{enumerate}
 \item $\mu_H^{\PS}$ is compactly supported if and only if the parabolic corank of $\G\cap H$ is zero.
\item
 $|\mu_{H}^{\PS}|<\infty$ if and only if
$\delta$ is bigger than the parabolic corank of $\G\cap H$.
\end{enumerate}
\end{thm}

As we show $\text{rank}(\G_\xi)-\text{rank}(\G\cap H)_\xi \le n-k$ for
all $\xi\in\Lambda_p(\G)\cap \partial(\bH^k) $, we have:

\begin{cor}\cite{OS} If $\delta>(n-k)$,
 then $|\mu_{H}^{\PS}|<\infty$.
\end{cor}

For instance, if we assume $\delta>(n-1)/2$, then $|\mu_H^{\PS}|<\infty$ whenever $k\ge (n+1)/2$.
 
\subsection{Distribution of $\G\ba \G H a_t$}
We first recall:
\begin{thm}\label{drs} 
 Suppose that $\G$ is a lattice in $G$ and that $H\cap \G$ is a lattice in $H$. In other words,
$|m^{\Haar}|<\infty$ and $|m^{\Leb}_H|<\infty$.
Then there exist $\eta_0>0$ (depending only on the spectral gap for $\G$) and $\ell\in \N$ such that
for any $\Psi\in C_c^\infty(\G\ba G)$, as $t \to \infty$,
$$\int_{\G\ba \G H} \Psi(ha_t)d\mu^{\Leb}_H(h) =\frac{|\mu^{\Leb}_H|}{ |m^{\Haar}|}\cdot  m^{\Haar}(\Psi) + O(\mathcal S_\ell(\Psi)
\cdot e^{-\eta_0 t}). $$
\end{thm}
In fact, the above theorem holds in much greater generality of any connected semisimple Lie group: the non-effective statement is due to \cite{DRS} (also see \cite{EM}).
For the effective statement, see \cite{DRS} for the case when $H\cap \G\ba H$ is compact and \cite{BeO} in general.

An analogue of Theorem \ref{drs} for discrete groups which are not necessarily lattices is given as follows:
\begin{thm} \cite{OS} \label{eq} Let $\G$ be Zariski dense with $|m^{\BMS}|<\infty$. Suppose $|\mu_H^{\PS}|<\infty$. 
Then for any $\Psi\in C_c(\G\ba G)$,
$$\lim_{t\to \infty} e^{(n-1-\delta)t} \int_{ h\in \G\ba \G H}\Psi(ha_t) d\mu^{\Leb}_H (h) =\frac{|\mu_H^{\PS}|}{|m^{\BMS}|} m^{\BR}(\Psi) .$$
\end{thm}

\begin{thm}\label{eq2} \cite{MO2} Let
$\G$ be a geometrically finite Zariski dense subgroup of $G$ with a spectral gap (e.g., $\delta>n-2$). Suppose  $|\mu_H^{\PS}|<\infty$. 
Then there exist $\eta_0>0$ and $\ell\in \N$ such that
for any $\Psi\in C_c^\infty(\G\ba G)$, as $t \to \infty$, 
$$ e^{(n-1-\delta)t} \int_{ h\in \G\ba \G H}\Psi(ha_t) d\mu^{\Leb}_H (h) =
\frac{|\mu_H^{\PS}|}{|m^{\BMS}|} m^{\BR}(\Psi)+ O(\mathcal S_\ell(\Psi)
\cdot e^{-\eta_0 t})  .$$
\end{thm}

 In the case when $\mu_H^{\PS}$ is compactly supported, we can show that
 there exists a compact subset $\mathcal O_H$ (depending on $\Psi$) of $\G\ba \G H$ such that for
any $t\in \br$, $\Psi(ha_t)=0$ for all $h\notin \mathcal O_H$. Therefore
 $\int\Psi(ha_t) d\mu^{\Leb}_H= \int_{\mathcal O_H}\Psi(ha_t) d\mu^{\Leb}_H$.
Now the assumption on $H$ being either symmetric or horospherical
ensures the wave front property of \cite{EM} which can be used to establish, as $t\to \infty$,
\be \label{tt} \int_{\mathcal O_H}\Psi(ha_t) d\mu^{\Leb}_H 
\approx \la a_t \Psi, \rho_{\mathcal O_H, \e}\ra_{L^2(\G\ba G)}\ee
where
$\rho_{\mathcal O_H, \e} \in C_c^\infty(\G\ba G)$ is an $\e$-approximation of $\mathcal O_H$.
Therefore the estimates on the matrix coefficients
in Theorems \ref{lo} and \ref{harmixing} can be used to establish Theorems \ref{eq} and \ref{eq2}.

The case when $\mu_H^{\PS}$ is not compactly supported turns out to be much more
intricate, the main reason being that
 we are taking the integral with respect to $\mu^{\Leb}_H$ 
as well as multiplying the weight factor $e^{(n-1-\delta)t}$ in the left hand side of Theorem \ref{eq},
whereas the finiteness assumption is made on the skinning measure $\mu^{\PS}_H$. In this case,
we first develop a version of thick-thin decomposition of the non-wandering set, that is,
$\mathcal W:=\{h\in \G\ba \G H: ha_t\in \text{supp}(\Psi)\}$, which resembles that of the support of $\mu_H^{\PS}$.
This together with \eqref{tt} takes care of the integral $\int \Psi_{ha_t}d\mu_H^{\Leb}(h) $
over a thick part as well as a very thin part of $\mathcal W$. What is left
is the integration over an intermediate range, which is investigated by comparing the two measures
$(a_t)_*\mu^{\PS}_H$ and $(a_t)_*\mu^{\Leb}_H$ via the transversal intersections
of the orbits $\G\ba \G Ha_t$ with the weak-stable horospherical foliations .
 A key reason
 that this approach works  is that 
these transversal intersections are governed by the topological properties
of the orbit $\G\ba \G Ha_t$, independent of the measures put on $\G\ba \G H$. 

In the special case of $n=2,3$ and $H$ horospherical,
Theorem \ref{eq2} was proved in \cite{LOA} by a different method.

\section{Distribution of $\G$ orbits in $H\ba G$ and Affine sieve}
\subsection{Distribution of $\G$ orbits in $H\ba G$}\label{oca}
Let $H$ be as in the previous section (i.e., symmetric or horospherical subgroup) and let
 $\{B_T\}$  be a family of compact subsets in $H\ba G$ which is getting larger and larger as $T\to \infty$.
 We assume that $[e]\G$ is discrete in $H\ba G$.
The study of the asymptotic of $\#([e]\G\cap B_T)$ can be now approached by
Theorems \ref{eq} and \ref{eq2} 
via the following counting function on $\G\ba G$:
$$F_T(g):=\sum_{\gamma\in H\cap \G\ba \G}\chi_{B_T}(\gamma g) $$
as $F_T(e)=\#([e]\G\cap B_T)$.
Depending on the subgroup $H$, we have $G=HA^+K$ or $G=HA^+K \cup HA^-K$.
For the sake of simplicity, we will consider the sets $B_T$ contained in $HA^+K$.
We have for $g=ha_tk$, $dm^{\Haar}(g)=\rho(t) d\mu^{\Leb}(h) dt dk$ where
$\rho(t)=e^{(n-1)t}(1+O(e^{-\beta t}))$ for some $\beta>0$.
Let $\mu$ be a $G$-invariant measure on $H\ba G$ normalized so that
$dm^{\Haar}= d\mu^{\Leb}_H\otimes d\mu $ locally.
Then, under the assumption
that $\G$ has a spectral gap and $|\mu_H^{\PS}|<\infty$,
we deduce from Theorem \ref{eq2}:
\begin{align}\label{ft2}
      &\la F_T, \Psi\ra =\int_{g\in B_T} \int_{h\in \G\ba \G H} \Psi (h g) d\mu^{\Leb}_H(h) d\mu(g) \notag \\
&=\int_{[e]a_tk\in B_T } \int_{h\in \G\ba \G H} \Psi^k (ha_t)  d\mu^{\Leb}_H(h) \rho(t) dtdk \notag \\
&= \int_{[e]a_tk\in B_T }
 \left(\tfrac{|\mu_H^{\PS}|}{|m^{\BMS}|} m^{\BR}(\Psi^k) e^{\delta t}  + O(e^{(\delta -\eta_0) t}) \right) dtdk \end{align}
for some $\eta_0>0$.

Therefore, similarly to the discussion in section \ref{cg},
if $F_T(e)$ can be {\it effectively approximated} by
$\la F_T, \Psi_\e \ra $ for an approximation of identity $\Psi_\e$ in $\G\ba G$,
a condition which depends on the regularity of the boundary of $B_T$ relative to the PS density,
then $F_T(e)$ can be computed by evaluating the integral \eqref{ft2} 
for $\Psi=\Psi_\e$ and by taking $\e$ a suitable power of $e^{-t}$.

For $H$ horospherical or symmetric,
 we have $G=HA^+K$ or $G=HA^+K\cup HA^-K$ (as a disjoint union except for the identity element)
where $A^\pm=\{a_{\pm t}: t\ge 0\}$.
\begin{dfn}\rm
Define a Borel measure $\mathcal M_{H\ba G}$ on $H\ba G$ as follows: for $\psi\in C_c(H\ba G)$,
\begin{equation*}\label{mde}
\mathcal M_{H\ba G} (\psi)=\begin{cases}\tfrac{ |\mu_{H}^{\PS}|}{|m^{\BMS}|} \int_{a_tk\in A^+K} \psi([e]a_tk)
  e^{\delta t} dt d\nu_o(k^{-1})&\text{if $G=HA^+K$}
\\ \sum \tfrac{ |\mu_{H,\pm}^{\PS}|}{|m^{\BMS}|}\int_{a_{\pm t}k\in A^\pm K}  \psi([e]a_{\pm t}k) 
e^{\delta t} dt d\nu_o(k^{-1})
 &\text{otherwise}, \end{cases}\end{equation*}
where $o\in \bH^n$ is the point fixed by $K$, $\nu_o$ is the right $M$-invariant measure on $K$, which
projects to the PS-measure $\nu_o$ on $K/M=\partial(\bH^n)$ and $\mu_{H,-}^{\PS}$ is the skinning measure on $\G\cap H\ba H$
in the negative direction.
\end{dfn}

\begin{dfn}\label{adm2} \rm 
  For a family $\{\B_T\subset H\ba G\}$ of compact subsets
 with $\mathcal M_{H\ba G}(\B_T)$ tending to infinity as $T\to \infty$, we say that $\{\B_T\}$
is {\it effectively well-rounded with respect to $\G$}
if there exists $p>0$ such that 
for all small $\e>0$ and $T\gg 1$:
$$\mathcal M_{H\ba G} (\B_{T,\e}^+ - \B_{T,\e}^-) = O(\e^p \cdot \mathcal M_{H\ba G} (\B_T)) $$  
where  $\B_{T,\e}^+=\B_{T}G_\e$ and $\B_{T,\e}^-= \cap_{g\in G_\e}\B_T g$.
\end{dfn}

In the next two theorems \ref{mc} and \ref{sec2},
we assume that
$\G'<\G$ is a subgroup of finite index with $H\cap \G=H\cap \G'$ and that both $\G$ and $\G'$ have
spectral gaps.

\begin{thm}\label{mc} \cite{MO2}
When $H$ is symmetric, we assume that $|\mu_H^{\PS}|<\infty$.
Suppose that $\{\B_T\}$ is effectively well-rounded with respect to $\G$. Then  
 for any $\gamma\in \G$, 
  there exists $\eta_0>0$ (depending only on a uniform spectral gap of $\G$ and $\G'$) such that 
$$\# ([e]\G'\gamma\cap \B_T )= \tfrac{ 1}{[\G:\G'] }  
\mathcal M_{H\ba G}(\B_T)  + O( \mathcal M_{H\ba G}(\B_T)^{1-\eta_0})$$
with the implied constant independent of $\G'$ and $\gamma\in \G$.
\end{thm}

For a given family $\{B_T\}$, understanding its effective well-roundedness 
can be the hardest part of the work in general.
However verifying the family of norm balls is
effectively well-rounded is manageable (see \cite[sec.7]{MO2}).
Consider an example of the family  $\{B_T=[e] A_T \Omega\}$ where $\Omega\subset K$ and $A_T=\{a_t: 0\le t\le T\}$.
In this case, it is rather simple to formulate the effective well-rounded condition:
 there exists $\beta'>0$ such that the PS-measure of the $\e$-neighborhood of $\partial(\Omega^{-1}M/M)$
is at most  of order $\e^{\beta'}$ for all small $\e>0$. As mentioned before,
 this holds when the boundary of $\Omega^{-1}M/M$ does not intersect the limit set $\Lambda(\G)$ (see \cite[sec. 7]{MO2} for a more general condition).

Then, setting
$\Xi(\G, \Omega):=\tfrac{|\mu_H^{\PS}|\cdot \nu_o(\Omega^{-1})}{\delta \cdot |m^{\BMS}|}$,
 one can deduce  from \eqref{ft2},
\begin{thm} \cite{MO2}  \label{sec2} Under the same assumption on $\mu_H^{\PS}$ and $\G'$,
 there  exists $\eta_0>0$ (depending only on a uniform spectral gap of $\G$ and $\G'$) such that for any $\gamma\in \G$,
\begin{equation*} \#([e]\G' \gamma \cap [e] A_T\Omega) =\frac{\Xi(\G,\Omega)}{[\G:\G']}{e^{\delta T}}
  +O(e^{(\delta -\eta)T}) \end{equation*}
with implied constant independent of $\G'$ and $\gamma$. \end{thm}
We note that in view of Theorem \ref{ft}, the non-effective version of this theorem holds
for any Zariski dense  $\G$ with $|m^{\BMS}|<\infty $ and $|\mu^{\PS}_{H}|<\infty$, as proved
in \cite{OS}.

\subsection{Affine sieve}\label{affine2}
For applications to an affine sieve, we consider the case when the homogeneous space
 $H\ba G$ is defined over $\z$. More precisely, we assume that
 $G$ is defined over $\z$ and acts linearly and irreducibly
on a finite dimensional vector space $W$ defined over $\z$ in such a way that
$G(\z) $ preserves $W(\z)$.
Let $w_0\in W(\z)$ be a non-zero vector such the stabilizer of $w_0$ is a symmetric subgroup
or the stabilizer of the line $\br w_0$ is a parabolic subgroup.
We set $V:=w_0G$. Let $\G$ be a geometrically finite Zariski dense subgroup of $G$ with a spectral gap,
which is contained in $G(\z)$.

Let $F$ be an integer-valued polynomial on the orbit $w_0\G$.
 Salehi-Golsefidy  and Sarnak \cite{SS}, generalizing \cite{BGS}, showed that for some $R>1$,
the set of ${\bf x}\in w_0\G$ such that $F({\bf x}) \text{ has at most $R$ prime factors}$
is Zariski dense in $V$.

For a square-free integer $d$, let $\G_d<\G$ be a subgroup which
 contains $\{\gamma\in \G: \gamma\equiv e \text{ mod }d\}$ and satisfies
 $\op{Stab}_{\G_d}(w_0)=\op{Stab}_\G(w_0)$. For instance, we can set $\G_d=\{\gamma\in \G: w_0\gamma \equiv w_0 \text{ mod }d \}$.
 
 We say the family $\{\G_d\}$ has a uniform spectral gap if $\sup_d s_0(\G_d)<\delta$ and $\sup_d n_0(\G_d)<\infty$
 (see Def. \ref{sngg} for notation).

 Salehi-Golsefidy and Varju \cite{SV}, generalizing \cite{BGS}, 
showed that the family of Cayley graphs of $\G/\G_d$'s (with respect to the
projections of a fixed symmetric generating set of $\G$) forms expanders
as $d$ runs through square-free integers with large prime factors.
 If $\delta>(n-1)/2$, the transfer property from the combinatorial spectral gap to the archimedean one established in \cite{BGS2}
(see also \cite{Kim}) implies that
 the family $\{\G_d:\text{ $d$ square-free} \}$ has a uniform {\it spherical} spectral gap.
 Together with the classification of $\hat G$, it follows that $\{\G_d\}$ admits a uniform spectral gap
if $\delta>(n-1)/2$ for $n=2,3$ or if $\delta>n-2$ for $n\ge 4$.
 
 For the following discussion, we assume that
 there is a finite set of primes $S$ such that
 the family $\G_d$ with $d$ square-free with no prime factors in $S$
 admits a uniform spectral gap.
  Then we can apply Theorem \ref{sec2} to $\G_d$'s. By a recent development on the affine linear sieve on homogeneous spaces
(\cite{BGS}, \cite{NS}), we are able to deduce: 
let $F$ be an integer-valued polynomial on $w_0\G$.
Letting $F=F_1 F_2 \cdots F_r$ be a factorization into irreducible
polynomials, assume that all $F_j$'s are integral on $w_0\G$ and distinct from each other.  
Let $\lambda$ be the log of the largest eigenvalue of $a_1$ on the $\br$-span of $w_0G$.

\begin{thm}\label{affine} \cite{MO2}
For any norm $\|\cdot \|$ on $V$,
\begin{enumerate}
 \item 
$\{{\bf x}\in w_0\G: \|{\bf x}\|<T,\; F_j({\bf x}) \text{ is prime for $j=1, \cdots, r$}\}\ll
\frac{T^{\delta/\lambda}}{(\log T) ^r} ;$
\item
there exists $R=R(F, w_0\G) \ge 1$ such that 
$$\{{\bf x}\in w_0\G: \|{\bf x}\|<T, \;F({\bf x}) \text{ has at most $R$ prime factors}\}\gg
\frac{T^{\delta/\lambda}}{(\log T) ^r} .$$
\end{enumerate}
\end{thm}

The significance of $T^{\delta/\lambda}$ in the above theorem is that for $H=\op{Stab}_G(w_0)$,
$$\mathcal M_{H\ba G} \{w\in w_0G: \|w\|<T\}\asymp \{{\bf x}\in w_0\G: \|{\bf x}\|<T\} \asymp T^{\delta/\lambda}.$$
In view of the above discussion, it is also possible to state Theorem \ref{affine} for
more general sets $B_T$, instead of the norm balls, which then provides a certain uniform distribution of almost prime vectors.

When $\Gamma$ is an arithmetic subgroup of a simply connected semi-simple 
algebraic $\q$-group $G$, and $H$ is a symmetric subgroup, the analogue of Theorem 6.6 was obtained in
\cite{BeO}.  Strictly speaking, \cite{BeO} is stated only for a fixed
group; however it is clear from its proof that the statement also holds
uniformly over its congruence subgroups with the correct main term.
 Based on this, one can use the combinatorial sieve to
obtain an analogue of Theorem 6.6 as it was done for a group
variety in \cite{NS}. Theorem 6.6 on lower bound for arithmetic was obtained
by Gorodnik and Nevo \cite{GN}  further assuming that $H\cap \Gamma$ is co-compact in $H$.

\section{Application to sphere packings}
In this section we will discuss counting problems for sphere packings in $\br^{n}$
as an application of an orbital counting problem for thin subgroups of $\SO(n,1)$. 
By a sphere packing in the Euclidean space $\br^{n}$ for $n\ge 1$, we simply mean
a union of (possibly intersecting) $n-1$-dimensional spheres; here
an $n-1$ dimensional plane is regarded as a sphere of infinite radius.

Fixing a sphere packing $\mathcal P$
in $\br^n$, a basic problem is to understand the asymptotic of
the number $\#\{S\in \mathcal P: \text{Radius}(S)>t\}$ 
or equivalently  $\#\{S\in \mathcal P: \text{Vol}(S)>t\}$ where $\text{Vol}(S)$ means the
volume of the ball enclosed by $S$. For the sake of brevity, we will simply refer $\text{Vol}(S)$ to
the volume of $S$.
We will consider this problem in a more general setting, that is, 
allowing the volume of the sphere $S$ to be computed in various conformal metrics.

\medskip

Let $U$ be an open subset of $\br^n$ and $f$ a positive continuous function on $U$.
A conformal metric associated to the pair $ (U,f)$  is a metric on $U$
of the form $f(x) dx$ where $dx$ is the Euclidean metric on $\br^n$.
Given a conformal metric $(U, f)$, we set
  $\op{Vol}_f(S):=\int_{B} f(x)^n dx$ where $B$ is the ball enclosed by $S$. When $S$ is a plane, 
we put $\op{Vol}_f(S)=\infty$. 

For any compact subset $E$ of $U$ and $t>0$,
set $$N_{t}(\P, f,E):=\#\{S\in \P: \op{Vol}_f(S)>t,\;\; E\cap S\ne \emptyset \}.$$
For simplicity, we will omit $\P$ in the notation $N_{t}(\P, f,E)$.
When $(U, f)=(\br^n, 1)$, i.e., the standard Euclidean metric,
we simply write $N_t(E)$ instead of $N_t(1, E)$.
We call $\P$ {\it locally finite} if $N_t(E)<\infty$ for all $E$ bounded.

\medskip

In order to approach the problem of the computation of the asymptotic of $N_t(f,E)$,
a crucial condition is that $\P$ admits enough symmetries of Mobius transformations of $\br^n$.
By the Poincar\'e extension, the group $\op{MG}(\br^n)$ of Mobius transformations of $\br^n$
can be identified with the isometry group of the upper half space 
 $\bH^{n+1}=\{(z, r): z\in \br^n, r>0\}$.
We assume that $\P$ is invariant under 
 a non-elementary discrete subgroup $\G$ of $G:=\op{Isom}^+(\bH^{n+1})$. 
As before, let $\delta$ denote the critical exponent of $\G$ and 
$\{\nu_x: x\in \bH^{n+1}\}$ a Patterson-Sullivan density for $\G$.

We recall from \cite{OS1}:
\begin{Def}[The $\G$-skinning size of $\P$] {\rm
For a sphere packing $\P$ invariant under $\G$,
define $0\le \op{sk}_\G({\P})\le \infty$ as follows:
$$\op{sk}_\G({\P}):=\sum_{i\in I} \int_{s\in \op{Stab}_{\G} (S_i^\dagger)\ba S_i^\dagger}  e^{\delta
\beta_{s^+}(o,s)}d\nu_{o}(s^+)$$
where $o\in \bH^{n+1}$, 
$\{S_i:i\in I\}$ is a set of representatives of $\G$-orbits in $\P$, and
$S_i^\dagger\subset \T^1(\bH^{n+1})$ is the set of unit normal vectors to
 the convex hull of $S_i$.
}\end{Def}


\begin{Def}\rm For the pair $(U,f)$, we define
a Borel measure $\omega_{\G, f}$ on $U$: for $\psi\in C_c(U)$ and for $o\in \bH^{n+1}$,
$$\omega_{\G, f}(\psi)=\int_{z\in U} \psi (z) f(z)^{\delta}
 e^{\delta \beta_z(o, (z,1)) }\; d\nu_{o}(z). $$
\end{Def}
Alternatively, we have a simple formula:
$$d\omega_{\G, f} = f(z)^{\delta} (|z|^2+1)^{\delta} d\nu_{e_{n+1}} $$
where $e_{n+1}=(0,\cdots,0,1)\in \bH^{n+1}$.

\begin{Ex}
\rm 
\begin{enumerate}
 \item For the spherical metric $(\br^n, \frac{2}{1+|z|^2})$  (also called the chordal metric)
on $\br^n$, $\op{Vol}_{f}(S)$ is the spherical volume of the ball enclosed by $S$ and 
$d\omega_{ \G, f}= 2^{\delta}\cdot d\nu_{e_{n+1}} .$
\item For the hyperbolic metric $(\bH^n=\{z\in \br^n: z_n>0\} , \frac{1}{z_{n}})$, 
 $\op{Vol}_{f}(S)$ is the hyperbolic volume of the ball enclosed by $S$ and 
$d\omega_{ \G, f}=   \frac{(1+|z|^2)^{\delta}}{z_{n}^{\delta}} d\nu_{e_{n+1}} .$
\end{enumerate}
\end{Ex}

\begin{Def}\rm
 By an \emph{infinite bouquet of tangent spheres glued at a point $\xi\in {\br^n}\cup\{\infty\}$}, 
we mean a union of two collections, each consisting of infinitely many pairwise internally tangent spheres
with the common tangent point $\xi$ and their radii tending to $0$,
 such that the spheres in each collection are externally tangent to the 
spheres in the other at $\xi$.
\end{Def}

We set $v_n:=\op{Vol}(B(0,1))$ to be the
 Euclidean volume of the unit ball $B(0,1):=\{x\in \br^n: \|x\|<1\}$; $v_n$ is equal to
 $\frac{(2\pi)^{n/2}}{2\cdot 4 \cdots n}$ if $n$ is even
and $\frac{2(2\pi)^{(n-1)/2} }{1\cdot 3 \cdots n}$ if $n$ is odd.

\begin{Thm} \label{m1} Let $\P$ be a locally finite sphere packing invariant under
a geometrically finite group $\G$
with finitely many $\G$-orbits. In the case
of $\delta\le 1$, we also assume 
that $\P$ has no infinite bouquet of spheres glued at a parabolic fixed point of $\G$.

Then for any conformal metric $(U,f)$ and for any compact subset $E$ of $U$ whose boundary has zero Patterson-Sullivan density,
  as $t\to 0$,
$$ \lim_{t\to 0} N_t(f,E) t^{\delta/n} =
\frac{\op{sk_\G}(\mathcal P)\cdot v_n^{\delta/n}\cdot \omega_{\G,f}(E)}
{\delta\cdot  |m^{\BMS}|}. $$
 \end{Thm}

The assumption on the non-existence of an infinite bouquet in $\P$ is to ensure that the
$\G$-skinning size for $\P$ is finite.

 In the case when $(U,f)=(\br^n,1)$, Theorem \ref{m1} was proved in \cite{OS1} for the case of $n=2$ and
and the proof given there extends easily for any $n\ge 2$ using the general equidistribution result in \cite{OS}.
To give an interpretation of Theorem \ref{m1} as a special case of
 the orbital counting problem discussed in the subsection \ref{oca},
fixing $H$ to be the stabilizer of a sphere $S_0$ in $\P$, we may think of $H\ba G$
as the space of all totally geodesic planes in $\bH^{n+1}$. Then a key point
is to describe a particular subset $B_t(E)$ in $H\ba G$ such that
$N_t(E)$ is same as $\# [e]\G \cap B_t(E)$.

The extension to a general conformal metric $(U,f)$ is possible basically due to the uniform continuity of $f$ on a compact subset $E$ and a covering
argument. We give a brief sketch as follows (the argument below was established in a discussion with Shah):
 denote by $Q_z(\eta)$ the cube
$\{z'\in \br^n: \max_{1\le i\le n} |z_i-z_i'|\le \eta\}$ 
centered at $z\in \br^n$ with radius $\eta$.

First, Theorem \ref{m1} for $f=1$, together with
 the uniform continuity of $f$ on $E$, implies
 that for any $\e>0$, there exists $\eta=\eta(\e)>0$ (depending
only on $E$ and $f$)  such that for any
cube $Q_z(\eta)$ centered at $z\in E$
\begin{equation}\label{lm}  \frac{N_t(f,Q_z(\eta)) t^{\delta/n} }{v_n^{\delta/n}}=
 (1+O(\e))\frac{\op{sk_\G(\mathcal P)}}{\delta \cdot |m^{\BMS}|} f(z)^{\delta} \omega_{\G}(Q_z(\eta)).  \end{equation}

Let $k$ be the minimal integer such that a $k$-dimensional sphere, say, $P$
charges a positive PS density. As the PS density is atom-free, we have $k>0$ and
the limit set $\Lambda(\Gamma)$ is contained
in $P$ (see \cite[Prop. 3.1]{Ro}).
We cover $E\cap P$ with cubes $\{Q_z(\eta):z\in I_{\eta}\}$  with disjoint interiors
for a finite subset $I_{\eta}$ of $E\cap P$.
As each cube is centered at a point of $P$ which is a sphere, the intersection of its 
boundary with $P$ is contained in a $(k-1)$-dimensional sphere for all small $\eta>0$.
 It follows that
 the boundary of each cube has zero PS density
by the minimality assumption on $k$.
Let $\mathcal C (\eta):=\{Q_z(\eta):z\in \tilde I_{\eta}\}$ be a covering of $E$
with $\tilde I_{\eta}\supset I_\eta$. Note that the boundary of each
cube in $\mathcal C(\eta)$ has  zero PS density.
 We can find compact subsets $E_\e^{\pm}$ of $U$ with
$E_\e^-\subset E\subset E_\e^+$ and a positive integer $m_\e$
so that
$\omega_{\G, f}(E_\e^+-E_\e^-)<\e$
and $E_\e^+$ (resp. $E$)
 contains all cubes centered at $E$ (resp. $E_\e^-$)
and of size less than $\eta$
possibly except at most $m_\e$ number of such cubes.

We may also assume that
$f(z)=(1+O(\e)) f(z')$ for all $z'\in Q_z(\eta)\in \mathcal C(\eta)$ where
the implied constant is uniform for all $z\in E$.
We may now apply Theorem \ref{lm} for $f=1$ to each cube in the covering of
$\mathcal C(\eta)$ to obtain that 

\begin{align*}
\frac{N_t(f,E) t^{\delta/n}}{v_n^{\delta/n}}& =
(1+O(\e)) \frac{\op{sk_\G(\mathcal P)}}{\delta \cdot |m^{\BMS}|}
 \sum_{z\in \tilde I_{\eta}}
 f(z)^{\delta} \omega_{\G}(Q_z(\eta))  \\
&= (1+O(\e)) \frac{\op{sk_\G(\mathcal P)}}{\delta \cdot |m^{\BMS}|}
  \omega_{\G, f}(E_\e^+) + O(\e)  \\
&= (1+O(\e))\frac{\op{sk_\G(\mathcal P)}}{\delta \cdot |m^{\BMS}|}
  \omega_{\G, f}(E) +O(\e) .
\end{align*}

As $\e>0$ is arbitrary, this proves Theorem \ref{m1}.

\section{On Apollonian circle packings}
\subsection{Construction} In the case of Apollonian circle packings in the plane $\br^2=\c$, Theorem \ref{m1} can be made more explicit
as the measure $\omega_{\G,f}$ turns out to be the $\delta$-dimensional Hausdorff measure with respect to $(U,f)$, restricted to
$\Lambda(\G)$.

We begin by recalling Apollonian circle packings, whose
construction is based on the following theorem of Apollonius of Perga:
\begin{thm}[Apollonius, 200 BC]
	Given $3$ mutually tangent circles in the plane (with distinct tangent points),
 there exist precisely two circles tangent to all three circles. 
\end{thm}



Consider four mutually tangent circles in the plane with distinct points of tangency. 
By the above Apollonius' theorem, one can add four new circles each of which is tangent to three of the given ones.
Continuing to repeatedly add new circles tangent to three of the previous circles,
we obtain an infinite circle packing, called an {\it Apollonian circle packing}.
 Figure \ref{f2} shows the first three generations of this procedure where each circle
is labeled with its curvature (the reciprocal of its radius).

\medskip

\begin{figure}
 \begin{center}
  \includegraphics[width=1.2in]{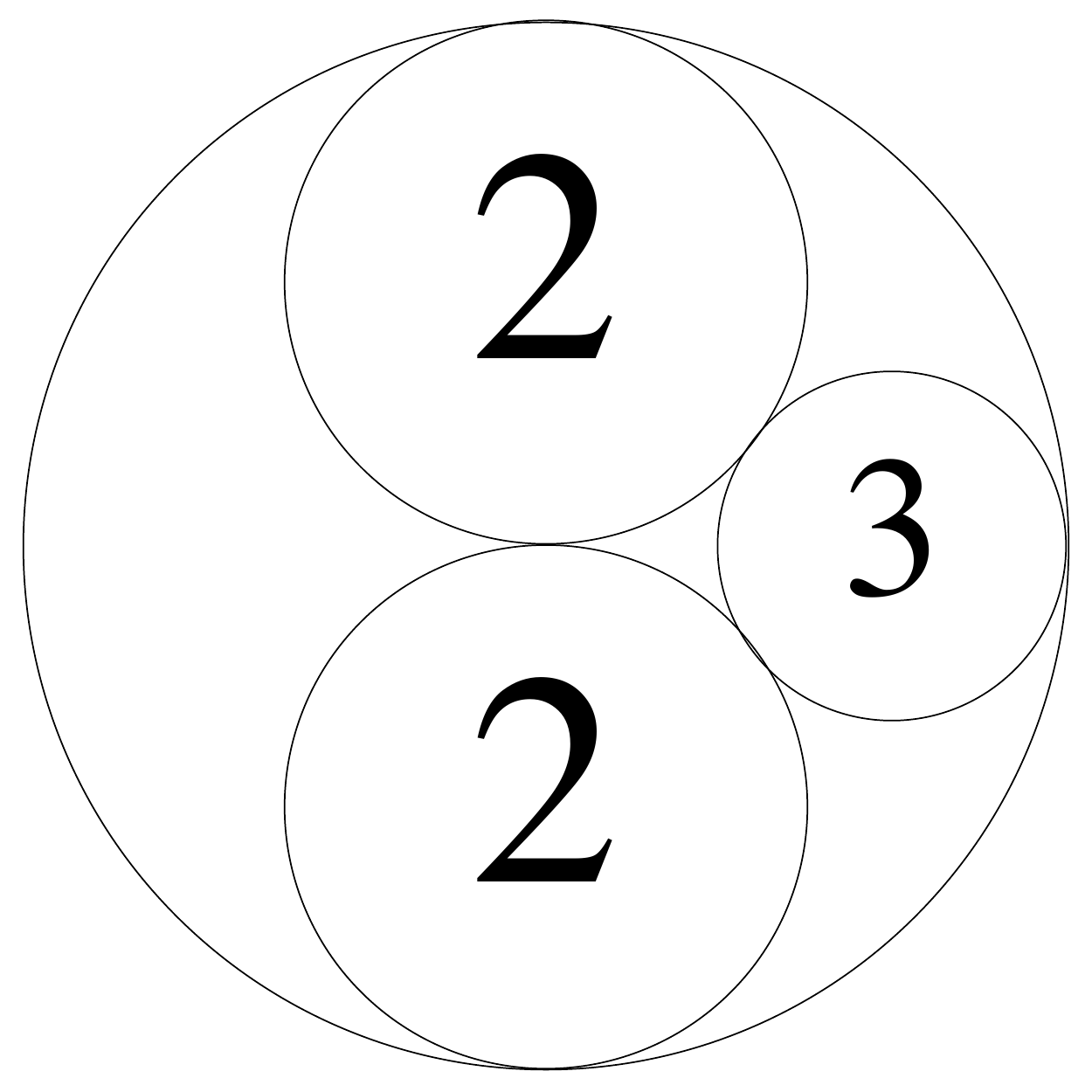} \includegraphics[width=1.2in]{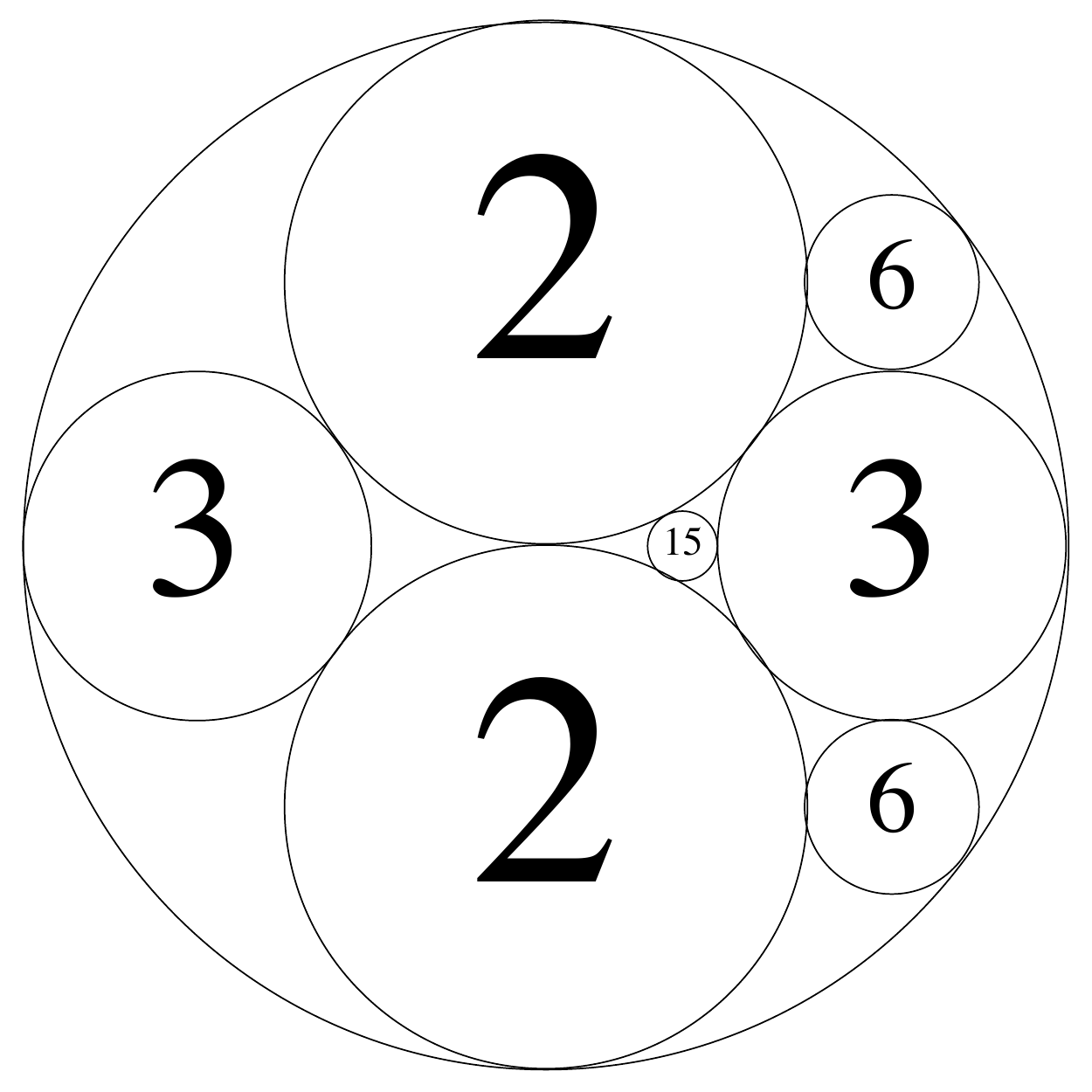} 
\includegraphics[width=1.2in]{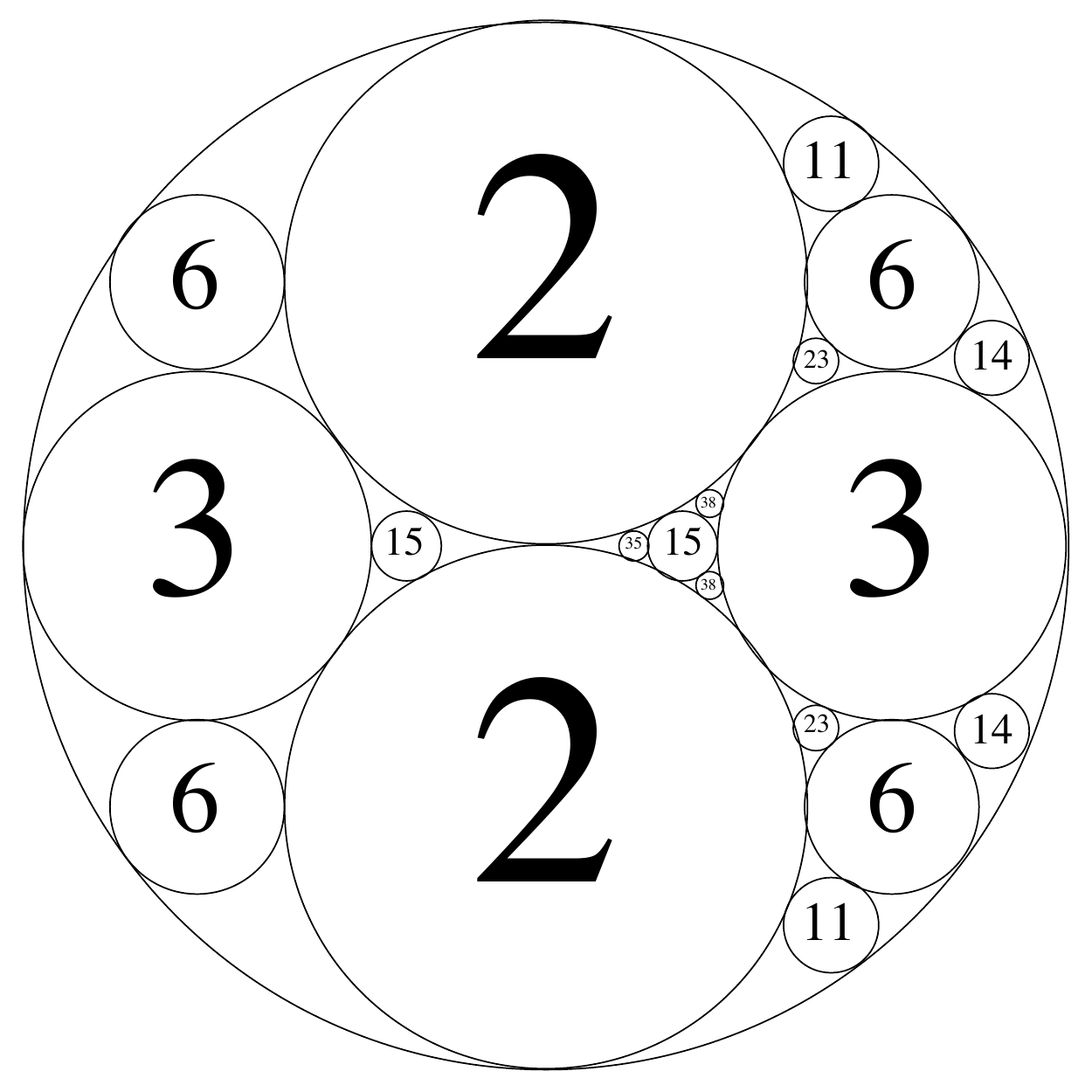}\includegraphics [width=1.2in]{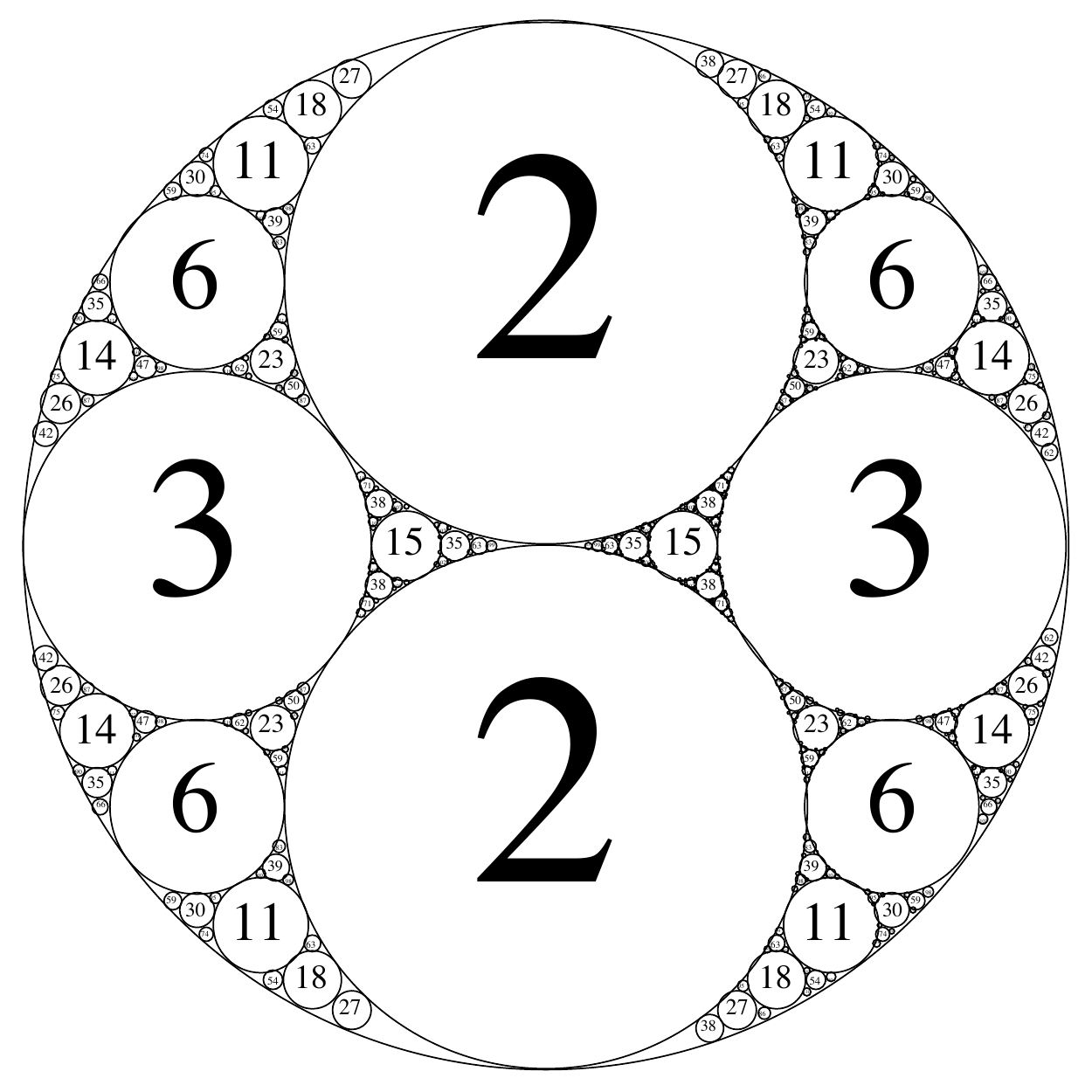}
\end{center}\label{f2}\end{figure}

\subsection{Apollonian packing and Hausdorff measure}
We start by recalling: \begin{dfn}\label{ha} \rm For $s>0$,
the $s$-dimensional  Hausdorff (also known as covering) measure ${\mathcal H}^{s}$
 of a closed subset $E$ of $\br^2$
is defined as follows: 
$$\mathcal H^{s}(E):=\lim_{\e \to 0} \; \inf \{\sum_{i\in I} \op{diam}(D_i)^{s}:
 E\subset \cup_{i\in I} D_i, \; \op{diam}(D_i) \le \e\} .$$
\end{dfn}

The Hausdorff dimension of $E$ is then given as 
 $$\op{dim}_\H (E):=\sup \{s: \H^s(E)=\infty\}=\inf\{s\ge 0: \mathcal H^s(E)=0\} .$$
For $s$ a positive integer, the $s$-dimensional Hausdorff measure is proportional to
the usual Lebesgue measure on $\br^s$.

For an Apollonian circle packing $\P$, the residual set $\Res(\P)$ is defined to be the closure of the union of all circles in $\P$.
Its Hausdorff dimension, say $\alpha$, is independent of $\P$ and known to be approximately
$1.30568(8)$ \cite{Mc}.

\begin{Thm}\label{ac}
 Let $\P$ be any Apollonian circle packing. 
For any conformal metric $(U,f)$ and for any compact subset $E\subset U$ with smooth boundary, 
we have
$$ \lim_{t\to 0} t^{\alpha/2}\cdot  \#\{C\in \P: \text{area}_f(C)>t, C \cap E \ne \emptyset\} = 
c_A\cdot \mathcal H_f^\alpha(\op{Res}(\P)\cap E) $$
where $c_A>0$ is independent of  $\P$ and
$d\mathcal H_{f}^\alpha (z)=f(z)^\alpha \cdot   d\mathcal H^\alpha(z) .$
\end{Thm}

The symmetry group $\G_{\P}:=\{g\in \PSL_2(\c): g(\P)=\P\}$ satisfies the following:
\begin{enumerate}
\item  $\G_{\P}$ is geometrically finite.
\item The limit set of $\G_{\P}$ coincides with $\Res(\P)$; in particular, its critical exponent is $\alpha$.
\item There are only finitely many $\G_{\P}$-orbits of circles in $\P$. \end{enumerate}

 Let
$\nu_{\P, j}$ denote the PS measure
viewed from $j=(0,0,1)\in \bH^3$ for the group $\G_{\P}$.
As $\G_{\P}$ has no rank $2$ parabolic limit points and $\alpha>1$,
 Sullivan's work \cite{Sullivan1984} implies that the $\alpha$-dimensional Hausdorff measure $\mathcal H^\alpha$
is a locally finite measure on $\Res(\P)$ and that
 $$\frac{1}{|\nu_{{\P},j}|} (|z|^2+1)^\alpha \; d\nu_{\P, j} =d\mathcal H^{\alpha}.$$
Therefore
 Theorem \ref{ac} is a special case of Theorem \ref{m1}.
Moreover the constant $c_A$ is given by
  $$c_A=\frac{\pi^{\alpha/2} \cdot \op{sk}_{\G_{\P}}(\mathcal P)\cdot |\nu_{\P, j}|}
{\alpha \cdot  |m^{\BMS}|}$$
for any Apollonian circle packing $\P$.
We propose to call $c_A$ the Apollonian constant.

\begin{Q}
\rm
Compute (or estimate) $c_A$!
\end{Q}

When $\P$ is a bounded Apollonian circle packing,
 the existence of the asymptotic formula $\#\{C\in \P: \text{area}(C)>t, \}
\sim c_\P\cdot  t^{\alpha/2}$ was first shown in \cite{KO} without an error term, and
 later in \cite{LOA} with an error term (see also \cite{V}). In view of \cite{OS1},
Theorem \ref{eq2} can be used to prove
an effective circle count in a compact region $E$ 
for general Apollonian packings, provided the boundary of $E$ satisfies a regularity property.

There is also a beautiful arithmetic aspect of Apollonian circle packings
which is entirely omitted in this article. We refer to  \cite{Sa}, \cite{OhICM}, \cite{BF}, \cite{BK}, etc.

\subsection{Apollonian sphere packing for $n=3$}
Given $n+1$ mutually tangent spheres in $\br^n$ with disjoint interiors,
it is known that
there is a unique sphere, called a {\it dual sphere}, passing through their points of tangency
and orthogonal to all $(n+1)$ spheres \cite[Thm 7.1]{GrahamLagarias2007}. 
Hence for $(n+2)$ mutually tangent spheres  with disjoint interiors
 $S_1, \cdots, S_{n+2}$ in $\br^n$, there are $(n+2)$  dual spheres,
say, $\tilde S_1, \cdots, \tilde S_{n+2}$. 
The Apollonian
group $\mathcal A=\mathcal A (S_1, \cdots, S_{n+2})$ 
is generated by
the inversions with respect to $\tilde S_i$, $1\le i\le n+2$.

Only for $n=2$ or $3$,
the Apollonian group $\mathcal A$ is a discrete subgroup of $\op{MG}(\br^n)$ \cite[Thm 4.1]{GrahamLagarias2007} and
in this case its orbit $\mathcal P:=\cup_{i=1}^{n+2} \mathcal A(S_i)$ consists of spheres with disjoint
interiors. For $n=2$, $\mathcal P$ is an Apollonian circle packing.
 For $n=3$,
$\P$ is called an Apollonian sphere packing. Note that $\mathcal A$ is geometrically finite
and that $\P$ is locally finite, as the spheres in
$\P$ have disjoint interiors,
 Hence Theorem \ref{m1} applies to $\P$.
The critical exponent of $\mathcal A$ for $n=3$ has been estimated to be $2.473946(5)$
in \cite{BPP}.

\subsection{Dual Apollonian Cluster Ensemble for any $n\ge 2$}
Given $(n+2)$ mutually tangent spheres  with disjoint interiors
 $S_1, \cdots, S_{n+2}$ in $\br^n$,
let $\mathcal A^*$ denote the group generated by the inversions with respect to
$S_i$, $1\le i\le n+2$.
The dual Apollonian group $\mathcal A^*$ is a discrete geometrically finite
subgroup of $\op{MG}(\br^n)$  for all $n\ge 2$ and
the orbit $\mathcal P:=\cup_{i=1}^{n+2} \mathcal A^*(S_i)$
is a sphere packing, in our sense, consisting of spheres nested
in $S_i$'s. We note that $\P$ is locally finite, as nested spheres are getting smaller and
smaller and hence Theorem \ref{m1} applies to $\P$.

\section{Packing circles of the ideal triangle in $\bH^2$} 
Consider  an ideal triangle $\mathcal T$ in $\bH^2$ i.e., a triangle whose sides are hyperbolic lines
connecting vertices on the boundary of $\bH^2$.
An ideal triangle exists uniquely up to hyperbolic congruences. Consider $\P(\mathcal T)$ the
circle packing of an ideal triangle by filling in largest inner circles.
The notation $\overline{\P(\mathcal T)}$ denotes the closure of $\P(\mathcal T)$
and $\text{area}_{\op{Hyp}}(C)$ is the hyperbolic area of the disk enclosed by $C$.

\begin{thm}[Packing circles of the ideal triangle]\label{pac} Let $\mathcal T$ be the ideal triangle of $\bH^2$.
Then
 $$ \lim_{t\to 0} \; t^{\alpha /2} \cdot \#\{C\in \P(\mathcal T): \text{area}_{\op{Hyp}}(C)>t\} = c_A
  \cdot\int_{\overline{\P(\mathcal T)}} y^{-\alpha}\;  d\mathcal H^{\alpha}(z)$$
where $c_A$ denotes the Apollonian constant.
\end{thm}

Fix the Apollonian circle  packing $\mathcal P_0$ generated by two vertical lines $x=\pm 1$ and
the unit circle $\{|z|=1\}$. The corresponding Apollonian group
$\G_0=\G(\P_0)$ is generated by the inversions with respect to horizontal lines $y=0$ and $y=-2i$
and the circles $\{|z-(\pm 1-i)|=1\}$. We set $\br^2_+:=\{x+iy:y>0\}$. 
Now for the conformal metric $(U,f)=(\br^2_+, 1/y)=\bH^2$,
we note that
$\{C\in \P(\mathcal T): \text{area}_{\op{Hyp}}(C)>t\}=\{C\in \P_0: \text{area}_{f}(C)>t,
C\cap \mathcal T\ne\emptyset\}$.

However Theorem \ref{pac} does not immediately follow from Theorem \ref{ac}
since the ideal triangle $\mathcal T$ is not a compact subset of $\bH^2$.

We need to understand the $\mathcal H^{\alpha}_f$-measure
of neighborhoods of cusps in the triangle for $f=1/y$.
For the next two theorems, consider a conformal metric $(\br^2_+, f)$.
\begin{Thm}\label{cm1}
If $f(x+iy)\ll y^{-k}$ for some real number $k>\alpha^{-1}$ with implied constant independent of $|x|\le 1$,
then for any $\eta>0$, $$\mathcal H_f^\alpha\{|x|\le 1, y>\eta\} <\infty .$$ 
Moreover for any Borel subset $E\subset \{|x|\le 1, y>\eta\} $ (not necessarily compact) with smooth boundary,
$$ \lim_{t\to \infty} \; t^{\alpha /2} \cdot \#\{C\in \P_0: \text{area}_{f}(C)>t, 
C \cap E \ne \emptyset\}\sim c_A  \cdot \mathcal H_f^\alpha (E) .$$
\end{Thm}
\begin{proof}
 It suffices to show the claim for $\eta=1$, since
 $\{|x|\le 1, \eta\le y \le 1\}$ is a compact subset. So we put $\eta=1$ and set $U_{R}:=\{|x|\le 1, y>R\}$.
Define $F_t(E):=\{C\in \P_0: \text{area}_{f}(C)>t, 
C \cap E \ne \emptyset\}$ and $E_n:=\{|x|\le 1, n\le y<n+1\}$.
Then $F_t(U_1)=\cup_{n\ge 1} F_t(E_n)$.

For $C\in F_t(E_n)$, $C-(n-1)i\in F_t(E_1)$ and
\begin{align*}\text{area}(C-(n-1)i)&=\int_{C-(n-1)i}f(z)^2 dz
 \\ & =\int_{C}f(z +(n-1)i))^2 dz\ll n^{-2k} \text{area}(C).
\end{align*}
Hence we get an injective map $F_t(E_n)$ to $F_{t n^{-2k}}(E_1)$
and hence for $R_0\ge 1$,
$$\# F_t(U_{R_0})=\sum_{n\ge R_0}\# F_{t n^{-2k} }(E_1).$$

By Theorem \ref{ac},
for $\e>0$, there exists $t_\e$ such that
for all $t<t_\e$,
$$\# F_t(E_1)\le t^{-\alpha /2} c_A (\mathcal H_f^\alpha(E_1)+\e)$$
and hence there exists $N_\e>1$ such that for all $t\le 1$ and $n>N_\e$,
$$\# F_{t n^{-2k} }(E_1)\le t^{-\alpha /2} n^{-k\alpha} c_A (\mathcal H_f^\alpha(E_1)+\e)$$
and hence 
$$\# F_t(U_{N_\e})t^{\alpha/2}\le (\sum_{n\ge N_\e}  n^{-k\alpha}) c_A (\mathcal H_f^\alpha(E_1)+\e).$$

Since
$ (\sum_{n\ge N_\e}  n^{-k\alpha})<\infty$ as $k\alpha>1$,
if $N_\e\gg 1$ is sufficiently large,
we can make $\# F_t(U_{N_\e})t^{\alpha /2}\le \e $
and $\mathcal H_f^\alpha(U_{N_\e})< \e$. This implies $\mathcal H_f^\alpha(U_{N_\e})< \infty$. 
Moreover,
\begin{multline*}\limsup \# F_t(U_1)t^{\alpha /2} =\limsup \#\sum_{1\le n<N_\e} F_t(E_n)t^{\alpha/2} +O(\e)\\ =
c_A\cdot \mathcal H_f^\alpha(U_1)+O(\e).\end{multline*}
As $\e>0$ is arbitrary,
we have $$\limsup \# F_t(U_1)t^{\alpha /2}=c_A \cdot \mathcal H_f^\alpha(U_1).$$
Similarly we have $\liminf \# F_t(U_1)t^{\alpha /2}=c_A \cdot \mathcal H_f^\alpha(U_1).$
Hence $$\lim \# F_t(U_1)t^{\alpha /2}=c_A \cdot \mathcal H_f^\alpha(U_1).$$ In the same way, we can deduce the claim for any Borel set $E$
using $\mathcal H_f^\alpha (E\cap U_{N_\e})=O(\e)$.
\end{proof}

We note that the boundary of $\br_+^2$ meets with $\op{Res}(\P_0)$ at three points
$1,-1,\infty$ and these three points are in one $\G_0$-orbit.
Note that $\G_0$ contains an element $\gamma_0$: $\gamma_0(x+yi)=x+(y+2)i$.

\begin{Thm} Let $(\br_+^2,f)$ be a conformal metric
such that
$f(x+iy)\asymp y^{-k}$ for $k\in \br$.  
\begin{enumerate}
 \item
If $\alpha ^{-1}< k < 2-\alpha^{-1}$,
we have $ \H_f^\alpha (\br_+^2) <\infty .$ 
\item If either $k<\alpha^{-1}$ or $k>2-\alpha^{-1}$, we have $ \H_f^\alpha (\br_+^2) =\infty .$
\end{enumerate}
\end{Thm}
\begin{proof} Let $\gamma (z)=\frac{\bar z-i}{\bar z-1-i}$. Then $\gamma\in \G_0$ and $\gamma(\infty)=1$.
Hence if we set $U_\eta:=\{|x|\le 1, y>\eta\}$, $\gamma(U_\eta)$ is a neighborhood of $1$ for all large $\eta>1$.
We will show that $\H_f^\alpha (\gamma(U_\eta))<\infty$
if $k<2-\alpha^{-1}$.

We have
\begin{align*}&
\H_f^\alpha (\gamma(U_\eta))=\int_{z\in \gamma(U_\eta)} f(z)^\alpha d\H^\alpha(z)\\ &=
\int_{w\in U_\eta} f(\gamma(w))^\alpha |\gamma'(w)|^{\alpha}  d\H^\alpha(w)
\end{align*}

Since $|\gamma'(w)|=|w-1+i|^{-2}$ and $\Im(\gamma(w))\asymp y^{-1}$, we have
$$f(\gamma(w)) |\gamma'(w)| \asymp y^{k-2}. $$
Hence by Theorem \ref{cm1},
if $2-k>\alpha^{-1}$,
$$\H_f^\alpha (\gamma(U_\eta))<\infty .$$
The remaining cases can be proved similarly and we leave them to the reader.
\end{proof}

Note that $\mathcal T:=\P_0\cap \br_+^2$ is a circle packing of the curvilinear triangle
made by largest inner circles.

As $\alpha>1$, Theorem \ref{pac} is a special case of the following:
\begin{Thm}
 Let $( \br_+^2,f)$ be a conformal metric
such that $f(x+iy)\asymp y^{-k}$ where $\alpha ^{-1}< k < 2-\alpha^{-1}$. 
Then for any Borel subset $E\subset \br_+^2$ (not necessarily compact)
with smooth boundary, we have
$$ \lim_{t\to \infty} \; t^{\alpha /2} \cdot \#\{C\in \mathcal T: \text{area}_{f}(C)>t, 
C \cap E \ne \emptyset\}\sim c_A  \cdot \mathcal H_f^\alpha (E) .$$
\end{Thm}


\end{document}